\documentclass[11pt, reqno]{amsart}

\usepackage{amssymb,latexsym,amsmath,amsfonts}
\usepackage{mathrsfs}
\usepackage{graphicx}
\usepackage[usenames]{color}

\numberwithin{equation}{section}

\allowdisplaybreaks

\theoremstyle{definition}
\newtheorem{definition}{Definition}[section]

\theoremstyle{remark}
\newtheorem{remark}[definition]{Remark}

 \theoremstyle{plain}
\newtheorem{theorem}[definition]{Theorem}
\newtheorem{result}[definition]{Result}
\newtheorem{lemma}[definition]{Lemma}
\newtheorem{proposition}[definition]{Proposition}

\newcommand{\eps}{\varepsilon}
\newcommand{\zt}{\zeta}
\newcommand{\zbar}{\overline{z}}
\newcommand{\wbar}{\overline{w}}

\newcommand{\poi}{\boldsymbol{{\sf P}}}
\newcommand{\Fej}{\boldsymbol{{\sf F}}}
\newcommand{\fL}{\mathscr{L}}
\newcommand{\tht}{\theta}

\newcommand{\bdy}{\partial}

\newcommand{\D}{\mathbb{D}}

\newcommand{\smoo}{\mathcal{C}}
\newcommand{\hol}{\mathcal{O}}
\newcommand{\I}{\boldsymbol{{\sf I}}}
\newcommand{\Ide}{\mathscr{I}}
\newcommand{\idl}{\mathscr{I}_{\!{\raisebox{-2pt}{$\scriptstyle \boldsymbol{{\sf X}}$}}}}
\newcommand\IX[1]{I_{{\raisebox{-1pt}{$\scriptstyle {{#1}}, \boldsymbol{{\sf X}}$}}}}
\newcommand\ide[1]{\mathscr{I}_{\!{\raisebox{-1pt}{$\scriptstyle {{#1}}, \boldsymbol{{\sf X}}$}}}}
\newcommand\bide[1]{\mathscr{I}^\bullet_{\!{\raisebox{-1pt}{$\scriptstyle {{#1}}, \boldsymbol{{\sf X}}$}}}}
\newcommand{\hinf}{H^\infty(\D^n)}
\newcommand{\Lb}{\mathbb{L}}
\newcommand{\al}{\mathfrak{A}}
\newcommand{\dal}{\mathscr{A}}

\newcommand{\Z}{\mathbb{Z}}
\newcommand{\nat}{\mathbb{N}}
\newcommand{\re}{{\sf Re}}

\newcommand{\bcdot}{\boldsymbol{\cdot}}
\newcommand\cis[1]{e^{i{#1}}}

\newcommand{\lrarw}{\longrightarrow}

\newcommand\newfrc[2]{{}^{\raisebox{-2pt}{$\scriptstyle {#1}$}}\!/_{\raisebox{2pt}{$\scriptstyle {#2}$}}}
\newcommand{\polcl}{\mathfrak{P}}
\newcommand\lprp[1]{{}^\perp{#1}}
\newcommand\rprp[1]{{#1}^\perp}
\newcommand\qnrm[1]{\boldsymbol{|\negmedspace|}{#1}\boldsymbol{|\negmedspace|}}
\newcommand\bv[1]{{#1}^\bullet}
\newcommand{\inttor}{\smallint\nolimits_{\raisebox{-2pt}{$\scriptstyle {\Tn}$}}}
\newcommand{\wes}{\text{weak}^{\raisebox{-1pt}{$\scriptstyle {\boldsymbol{*}}$}}}
\newcommand{\pj}{x^{(j)}}
\newcommand{\Kdel}{\boldsymbol{\delta}}
\newcommand{\weak}{{\rm wk}^{\raisebox{-1pt}{$\scriptstyle {\boldsymbol{*}}$}}(\boldsymbol{{\sf I}}, \infty)}
\newcommand\spProj[1]{\Pi_{{#1}, \boldsymbol{{\sf X}}}}
\newcommand{\pee}{p^{\raisebox{-1pt}{$\scriptstyle (\varepsilon)$}}}
\newcommand{\spee}{p^{(\varepsilon)}}

\newcommand{\Cn}{\mathbb{C}^n}
\newcommand{\C}{\mathbb{C}} 
\newcommand{\R}{\mathbb{R}}

\newcommand{\Tn}{\mathbb{T}^n}

\begin{document}

\title[Pick interpolation on the polydisc]{Pick interpolation on the polydisc: small \\ families of sufficient kernels}

\author{Gautam Bharali}
\address{Department of Mathematics, Indian Institute of Science, Bangalore 560012, India}
\email{bharali@math.iisc.ernet.in}

\author{Vikramjeet Singh Chandel}
\address{Department of Mathematics, Indian Institute of Science, Bangalore 560012, India}
\email{abelvikram@math.iisc.ernet.in}

\thanks{V.\,S.\;Chandel is supported by a scholarship from the Indian Institute of Science}


\begin{abstract}
{We give a solution to Pick's interpolation problem on the unit polydisc in $\Cn$, $n\geq 2$, by characterizing
all interpolation data that admit a $\D$-valued interpolant, in terms of a family of positive-definite
kernels parametrized by a class of polynomials. This uses a duality approach that has been associated
with Pick interpolation, together with some approximation theory. Furthermore, we use
duality methods to understand the set of points on the $n$-torus at which the boundary values of a given solution
to an {\em extremal} interpolation problem are not unimodular.}
\end{abstract}
\maketitle

\section{Introduction, some preliminaries, and a statement of results}\label{S:intro}

The interpolation problem referred to in the title is as follows: 
\begin{itemize}
\item[$(*)$] Let $X_1,\dots,X_{N}$ be distinct points in the polydisc $\D^n$ and let
$w_1,\dots,w_{N}\in\D$. Find a necessary and sufficient condition on the data $\{(X_j, w_j) : 1\leq j\leq N\}$
such that there exists a holomorphic function $F : \D^n\lrarw \D$ satisfying $F(X_j)=w_j, \ j=1,\dots,N$.
\end{itemize}
Here, and elsewhere in this paper, $\D$ denotes the open unit disc with centre $0\in \C$. 
We begin by discussing some of the ideas and
results that have influenced our theorems below (although our overview of those ideas will
be slightly ahistorical). We must begin by stating that the ideas alluded to have a close
connection to the work of Cole, Lewis and Wermer \cite{coleLewisWermer:PcuavNi92}
(also see \cite{coleWermer:PivNihc94} by Cole and Wermer) on the existence of interpolants in a given
uniform algebra for an interpolation problem between its maximal ideal space and $\D$.
\smallskip

At the heart of the works
\cite{coleLewisWermer:PcuavNi92} and \cite{coleWermer:PivNihc94} is a method, which goes back to
Sarason \cite{sarason:giH67}, of representing
the quotient of a uniform algebra by a closed ideal as an algebra of operators on some Hilbert space.
It turns out that a formula for the quotient norm in such a setting\,---\,which derives from the representation
alluded to\,---\,can be transported to the setting of {\em dual algebras} and their quotients by
{\em weak${}^{{\boldsymbol{*}}}$ closed
ideals}. In \cite{mccullough:NPtida96}, McCullough provides such a formula. He further uses the insights gained in
proving this formula in such a way as to {\em also} address the existence of interpolants in $\hinf$ for the
problem $(*)$.
\smallskip

Let us elaborate upon the phrase ``dual algebra''.
Given a complex, separable Hilbert space $H$, let $\mathcal{B}(H)$ be the space of bounded
operators on $H$. It is known that the dual of the space of trace class operators of $H$ is
isometrically isomorphic to $\mathcal{B}(H)$ (endowed with the operator-norm topology). Via this
isomorphism, one can make sense of the $\wes$ topology on $\mathcal{B}(H)$. A unital subalgebra
$\dal$ of $\mathcal{B}(H)$ is called a {\em dual algebra} if it is $\wes$ closed. Our interest in dual algebras
stems from the fact that $\hinf$\,---\,the class of all bounded holomorphic functions on $\D^n$\,---\,is a
dual algebra. Hence, let us specialize to $\D^n$. Write:
\begin{align}
 A(\D^n)\,&:=\,\smoo(\overline{\D^n}; \C)\cap\hol(\D^n), \notag \\
 \Tn\,&:=\,(\bdy\D)^n \varsubsetneq \bdy(\D^n) \quad \text{and} 
 \quad m\,=\,\text{the normalized Lebesgue measure on $\Tn$}. \label{E:Leb-meas}
\end{align}
Recall that the classical Hardy space $H^2(\Tn)$ is the closure in $\Lb^2(\Tn, dm)$ of
$\left.A(\D^n)\right|_{\Tn} := \{\left.f\right|_{\Tn} : f\in A(\D^n)\}$. The space of all multipliers
preserving $H^2(\Tn)$ is $\hinf$. (The functions in $H^2(\Tn)$ and $\hinf$
have different domains of definition, but we assume
that readers know how this apparent problem is dealt with\,---\,and refer them to Section~\ref{SS:charac_weak}
if they don't.) Viewed as a subalgebra of $\mathcal{B}(H^2(\Tn))$, it is known that $\hinf$ is a dual algebra.
In view of the discussion above, 
with $H = H^2(\Tn)$,
it is meaningful to talk about the $\wes$ closure of a subalgebra of $\hinf$.
\smallskip

We now have almost all the background needed to present our first theorem, and to introduce a
result that has strongly influenced this theorem. We first fix some notation. We will
always use $\al$ to denote a uniform subalgebra of $A(\D^n)$. Given $g\in \Lb^2(\Tn, dm)$, we shall
set
\[
 \al^2(g)\,:=\,\text{the closure of $\left.\al\right|_{\Tn}$ in $\Lb^2(\Tn, |g|^2dm)$}.
\]
The following spaces associated to $\al$ are very useful in the discussion of Pick interpolation in higher dimensions:
\begin{align}
 \lprp{\al}\,&:=\,\Big\{f\in \Lb^1(\Tn) : \inttor\!\psi f dm = 0 \ \text{for each} \ \psi\in \al\Big\},
 \label{E:lprp_A} \\
 \dal(\al)\,&:=\,(\text{the closure of $\left.\al\right|_{\D^n}$ in the topology of local unif.\;convergence})\cap\hinf.
 \label{E:luClo_A}
\end{align}
Furthermore, we need a definition. (We shall abbreviate $\Lb^p(\Tn, dm)$ to $\Lb^p(\Tn)$, $p = 1, 2, \infty$.)

\begin{definition}
Let $\al$ be a uniform subalgebra of $A(\D^n)$. We say that $\al$ {\em has a tame pre-annihilator} if
$(\smoo(\Tn; \C)\cap\lprp{\al})$ is dense in $\lprp{\al}$ in the $\Lb^1(\Tn)$-norm.
\end{definition}

Theorem~\ref{T:interp-Char} below is strongly motivated by the following result of McCullough. We shall
paraphrase it for the case of the polydisc $\D^n$, since this is the representative case, and the argument
for the set-up in \cite[Theorem~5.12]{mccullough:NPtida96} follows, after a few adjustments, nearly verbatim the
argument in the case of $\D^n$.

\begin{result}[paraphrasing {\cite[Theorem~5.12]{mccullough:NPtida96}} for the case of $\D^n$, and
$m$ as in {\eqref{E:Leb-meas}}]\label{R:McC}
Let $X_1,\dots,X_{N}$ be distinct points in $\D^n$, $n\geq 2$, and let $w_1,\dots,w_{N}\in\D$. Fix a uniform algebra
$\al\subseteq A(\D^n)$ having a tame pre-annihilator. Furthermore assume that
\begin{itemize}
 \item[$(a)$] $\al$ is approximating in modulus, and
 \item[$(b)$] $\overline{K_{\al}}(X_j, \bcdot) \in \left.\al\right|_{\Tn}$ for each $j = 1, 2,\dots, N$,
\end{itemize}
where $K_{\al}(x, \bcdot)$, $x\in \D^n$, is the Szeg{\H{o}} kernel associated with the Hilbert space $\al^2(1)$.
Then, there exists a function $F\in (\Lb^\infty(\Tn)\cap \al^2(1))$ with $\sup_{\D^n}|F|\leq 1$ and such that
the Poisson integral $\poi[F]$ satisfies
$\poi[F](X_j) = w_j$, for each $j = 1,\dots, N$, if and only if the matrices
\begin{equation}\label{E:Pick-sys}
 \left[(1 - w_j\wbar_k)\big\langle K_{\al, \psi}(X_j, \bcdot), K_{\al, \psi}(X_k, \bcdot)\big\rangle_{\al^2(\psi)}
 \right]_{j,\,k = 1}^N\,\geq\,0,
\end{equation}
for each $\psi\in \al$ such that $|\psi| > 0$ on $\Tn$, where $K_{\al,\psi}(x, \bcdot)$, $x\in \D^n$,
is the Szeg{\H{o}} kernel associated with the Hilbert space $\al^2(\psi)$.
\end{result}

\noindent{We refer the reader to the beginning of Section~\ref{S:funcAnal} for a discussion of the term
``Szeg{\H{o}} kernel associated to a Hilbert space'', and of the notation we follow. A uniform
subalgebra $\al\subseteq A(\D^n) $ is said to be {\em approximating in modulus} if for each non-negative
function $g\in \smoo(\Tn; \C)$ and each $\eps > 0$, there exists a $\psi\in \al$ such that
$\sup_{\,\Tn}|\,|\psi| - g| < \eps$.}
\smallskip

In its full generality, \cite[Theorem~5.12]{mccullough:NPtida96} is an interpolation theorem of the
Cole--Lewis--Wermer type. In its paraphrasing as Result~\ref{R:McC}, it is {\em very}
interesting because it solves the problem $(*)$, with interpolants belonging to the Schur class. Moreover, it does
so by providing us with an easier to understand and smaller family of kernels\,---\,i.e., those that feature
in \eqref{E:Pick-sys}\,---\,necessary and sufficient for the existence of an interpolant than those appearing in 
\cite{coleLewisWermer:PcuavNi92, coleWermer:PivNihc94}. (We shall not elaborate any further: interested
readers are referred to \cite[Proposition~5.9]{mccullough:NPtida96}.)
 It is not possible, when
$n\geq 2$, to replace the family of Pick matrices in \eqref{E:Pick-sys} with a single matricial condition as in Pick's
well-known solution to $(*)$ for $n = 1$. Yet, the contrast between
Pick's result and the situation when $n\geq 2$ is a constant
stimulus to finding a smaller and/or more explicitly defined family of kernels that are necessary and sufficient for the
existence of an interpolant. Indeed, this was our primary motivation for re-examining the proof of Result~\ref{R:McC}
and for the following (in this paper, $D(a; r)$ will denote the open disc of radius $r > 0$ with centre $a\in \C$):

\begin{theorem}\label{T:interp-Char}
Let $X_1,\dots,X_{N}$ be distinct points in $\D^n$, $n\geq 2$, and let $w_1,\dots,w_{N}\in\D$. Let $\dal$ be a
weak${}^{{\boldsymbol{*}}}$ closed subalgebra of $\hinf$ such that $\dal = \dal(\al)$ for some uniform subalgebra
$\al\subseteq A(\D^n)$ having a tame pre-annihilator. Fix an integer $R\geq 1$, and define
\[
 \polcl(R)\,:=\,\left\{p\in \C[z_1,\dots, z_n] : p^{-1}\{0\}\cap\overline{D(0; R)}^n = \varnothing\right\}.
\] 
There exists a function $F\in \dal(\al)$ such that $F: \D^n\lrarw \D$ and $F(X_j) = w_j$, for each $j = 1,\dots, N$,
if and only if the matrices
\begin{equation}\label{E:Pick-sys2}
 \left[(1 - w_j\wbar_k)\big\langle K_{\al,\,p}(X_j, \bcdot), K_{\al,\,p}(X_k, \bcdot)\big\rangle_{\al^2(p)}
 \right]_{j,\,k = 1}^N\,\geq\,0 \; \; \; \text{for each $p\in \polcl(R)$},
\end{equation}
where $K_{\al,\,p}(x, \bcdot)$, $x\in \D^n$,
is the Szeg{\H{o}} kernel associated with the Hilbert space $\al^2(p)$.
\end{theorem}

\begin{remark}\label{rem:features}
The hypothesis on $\dal$ above holds true for $\dal = \hinf$. It is
well known that $\hinf = \dal(A(\D^n))$ (as per our notation in \eqref{E:luClo_A}). We refer
the reader to the end of Section~5 in \cite{mccullough:NPtida96} for a demonstration that
$A(\D^n)$ has a tame pre-annihilator. See the first paragraph of Section~\ref{S:non-mod1Set_proof}
for a gist of that discussion. In short, Theorem~\ref{T:interp-Char} provides new information
even for the basic problem $(*)$. Secondly, for both
the classical problem $(*)$ and when $\al\varsubsetneq A(\D^n)$ we give a much more explicit
family of kernels than Result~\ref{R:McC} that are sufficient for interpolation. Indeed, we
see that there are {\em progressively smaller} families of kernels that are sufficient for interpolation.
Lastly, Theorem~\ref{T:interp-Char} is a result of Cole--Lewis--Wermer type, characterizing the existence
of interpolants in a variety of unital $\wes$ closed subalgebras of $\hinf$.
\end{remark}

The last sentence of Remark~\ref{rem:features} needs some explanation. Theorem~\ref{T:interp-Char}
suggests that $\dal(\al)$, as defined in \eqref{E:luClo_A}, is $\wes$ closed. In fact, with {\em no
further conditions} on $\al \subseteq A(\D^n)$, we can calculate its $\wes$ closure as follows:

\begin{proposition}[see Proposition~\ref{P:closure_subalg} below]\label{P:closure_wes}
Let $\al$ be a uniform
subalgebra of $A(\D^n)$. Its weak${}^{{\boldsymbol{*}}}$ closure (with $\al$ viewed
as embedded in $\mathcal{B}(H^2(\Tn))$ with the weak${}^{{\boldsymbol{*}}}$ topology) is
$\dal(\al)$.
\end{proposition}
  
Before we introduce our next theorem, we ought to mention that the representation, alluded to above,
of the quotient of a uniform algebra by a closed ideal as an algebra of operators on some Hilbert space
was first proved for $A(\D)$ by Sarason in \cite{sarason:giH67}. His approach to
Pick interpolation has been very influential. That approach led to Agler's solution of $(*)$ for $n = 2$:
see \cite{agler:unpub88} (see also the articles \cite{ball:UcrkHsN-Pisv98} by Ball--Trent and 
\cite{agmac:N-Pib99} by Agler--McCarthy).
There have been a number of articles, based on largely functional-analytic ideas, in the last two decades
that have dwelt on the problem $(*)$: we refer the reader to the works listed in the
bibliography of \cite{hamilton:Pisv13}. The latter work, we must mention,
addresses\,---\,using a result of Bercovici--Westwood \cite{bercWest:ffp92}\,---\,the problem of characterizing
the existence of interpolants in an arbitrary unital $\wes$ closed subalgebra of $\hinf$.
Our proof of Theorem~\ref{T:interp-Char}
also relies, to an extent, on
some of those ideas (and is influenced by \cite{mccullough:NPtida96}). However, at one
crucial juncture\,---\,when we introduce the family $\polcl(R)$\,---\,we revisit some hands-on computations
involving the uniform algebra $A(\D^n)$. Additionally, Proposition~\ref{P:closure_wes} plays an essential
role in the proof.
\smallskip

Our next result is aimed at understanding the functions that interpolate the data
$\{(X_j, w_j) : 1\leq j\leq N\}$ for which the interpolation problem $(*)$ is extremal. We say that
the problem $(*)$\,---\,given the data $\{(X_j, w_j) : 1\leq j\leq N\}$\,---\,is {\em extremal} if it admits
an interpolant $F$ for these data with $\sup_{\D^n}|F| = 1$ but admits no interpolant of
sup-norm less than $1$.
\smallskip

The specific form of Theorem~\ref{T:non-mod1Set} below is motivated, in part, by a result of
Amar and Thomas \cite{amarThomas:fimun00} (see below), and by the fact that
the generic extremal problem for the bidisc, and with $N = 3$, has a unique solution that is a rational inner
function\,---\,see \cite[Theorem~12.13]{aglerMcCarthy:PiHfs02}. Some interesting results on the extremal
problem in higher dimensions, but still with $N = 3$, were obtained recently by Kosi{\'n}ski \cite{kos:3pNPpp15}.
Little is currently known when $N\geq 4$. It is not even known whether,
for a generic extremal problem, there exists an interpolant that (generalizing the situation in the bidisc) is an
inner function. A bounded holomorphic function $f$ on $\D^n$ is called an {\em inner function} if
the values of the radial boundary-value function $\bv{f}$, defined as
\begin{equation}\label{E:rbv}
 \bv{f}(\zt)\,:=\,\lim_{r\to 1^-}f(r\zt) \; \; \; \text{(for $m$-a.e. $\zt\in \Tn$)},
\end{equation}
are unimodular $m$-a.e. on $\Tn$. We recall here that the fact that the limit on the right-hand side of
\eqref{E:rbv} exists $m$-a.e. on $\Tn$ is the extension of a well-known theorem of Fatou to higher
dimensions (see Section~\ref{SS:charac_weak} for more details). 
\smallskip

Amar and Thomas use the phrase ``all the points of $\{X_j : 1\leq j\leq N\}$ are active constraints'' 
to refer to a generic extremal problem on $\D^n$. We shall not define this term here; the reader is referred
to \cite[Section~0]{amarThomas:fimun00} for a definition. The authors are interested in the nature of
the maximum modulus set $M(\phi)$ of an interpolant $\phi$ for a given extremal problem. To be precise:

\begin{result}[paraphrasing {\cite[Theorem~1]{amarThomas:fimun00}} for the case of the polydisc]\label{R:AT}
Let $X_1,\dots,X_{N}$, distinct points in $\D^n$, $n\geq 2$, and $w_1,\dots,w_{N}\in\D$ be data for an extremal
Pick interpolation problem on $\D^n$. Let $\phi$ be any interpolant in the Schur class. Write
\[
 M(\phi)\,:=\,\{\zt\in \Tn : \limsup\nolimits_{\D^n\ni z\to \zt}|\phi(z)| = 1\}.
\]
Let $[M(\phi)]^{\wedge}_{A(\D^n)}$ denote the $A(\D^n)$-hull of $M(\phi)$. 
If all the points of $\{X_j : 1\leq j\leq N\}$ are active constraints, then
$[M(\phi)]^{\wedge}_{A(\D^n)}\supset \{X_j : 1\leq j\leq N\}$. In general,
$[M(\phi)]^{\wedge}_{A(\D^n)}\cap\{X_j : 1\leq j\leq N\}\neq \varnothing$.
\end{result}

The result above describes, in some sense, the structure of $M(\phi)$. A natural question that arises from the
discussion prior to Result~\ref{R:AT} is how close the interpolant $\phi$ is to an inner function.
This entails studying the size of the set $\{\zt\in \Tn : |\bv{\phi}(\zt)| = 1\}$.  Result~\ref{R:AT}
does not quite provide this information and, furthermore, we have the difficulty that 
\[
 M(\phi)\,\supseteq\,\{\zt\in \Tn : |\bv{\phi}(\zt)| = 1\}.
\]
However, some of the tools used in our proof of Theorem~\ref{T:interp-Char} can be used to obtain
information on the set on the right-hand side above. To be more precise, we show that
if $\{\zt\in \Tn : |\bv{\phi}(\zt)| = 1\}$ is not of full measure, then the set
$\Tn\setminus \{\zt\in \Tn : |\bv{\phi}(\zt)| = 1\}$ is constrained in a rather specific fashion. Before
we can state this theorem, we need the following

\begin{definition}
Let $X$ be a real-analytic manifold. A set $S\subseteq X$ is called a {\em semi-analytic set} if for each
point $p\in S$, there exists an open set $U_p\ni p$ and functions $f_{jk}\in \smoo^{\omega}(U_p; \R)$,
$j = 1,\dots, \mu$, $k = 1,\dots, \nu$, such that
\[
 S\cap U_p = \bigcup\nolimits_{1\leq j\leq \mu}\bigcap\nolimits_{1\leq k\leq \nu}S_{jk},
\]
where each $S_{jk}$ is either $\{x\in U_p : f_{jk}(x) = 0\}$ or $\{x\in U_p : f_{jk}(x) > 0\}$.
\end{definition}

We are now in a position to state our next theorem.

\begin{theorem}\label{T:non-mod1Set}
Let $X_1,\dots, X_N$ be distinct points in $\D^n$, $n\geq 2$, and let $w_1,\dots, w_N\in \D$.
Assume that $(X_1,\dots, X_N;\,w_1,\dots, w_N)$ are data for an extremal Pick interpolation
problem. Let $\phi$ be any interpolant in the Schur class, and let $\bv{\phi}$ denote the radial
boundary-value function of $\phi$. Then, the set $\{\zt\in \Tn : |\bv{\phi}(\zt)| < 1\}$ is contained
in the disjoint union $N\sqcup S$, where $N$ is a set of zero Lebesgue measure and
$S$ is the inner limit of a sequence of proper
semi-analytic subsets of $\Tn$.
\end{theorem}

The proofs of Theorems~\ref{T:interp-Char} and \ref{T:non-mod1Set} will be presented in
Sections~\ref{S:interpCond} and \ref{S:non-mod1Set_proof}, respectively. The proof of
Proposition~\ref{P:closure_wes} will be the subject of Section~\ref{SS:charac_weak}.
However, we shall
need a few standard facts and a couple of essential propositions before we can give these proofs.
Section~\ref{S:funcAnal} will be devoted to matters that are primarily functional-analytic in character.
Section~\ref{S:funcThr} will be devoted to matters
pertaining to function theory in several complex variables.
\medskip

\section{On duality and the weak${}^{{\boldsymbol{*}}}$ topology}\label{S:funcAnal}

This section is intended to present several results, which are primarily functional-analytic 
in character, that we will need in the proofs of our theorems. Along the way, we shall explain
a few terms that had appeared in Section~\ref{S:intro} and whose discussion had been deferred.

\subsection{Szeg{\H{o}} kernels associated to Hilbert spaces on $\boldsymbol{\Tn}$}\label{SS:intro_szeg_ker}

We adopt the notation introduced in Section~\ref{S:intro}. Let $\al$ be a
uniform subalgebra of $A(\D^n)$, $g\in\Lb^{\infty}(\Tn)$ be such that $|g| > c_g$ for some constant
$c_g > 0$, and
let $\al^{2}(g)$ be as defined in Section~\ref{S:intro}. By construction, $\al^{2}(g)$
is a separable Hilbert space with the inner product
\[
\langle\psi,\varphi\rangle_g\,:=\,\inttor\!\psi\bar{\varphi}|g|^2 dm.
\]
In this paper, for any $\varphi\in\Lb^{1}(\Tn)$, we shall write
\[
\poi[\varphi]\,:=\,\text{the Poisson integral of $\varphi$}.
\]

By the properties of $g$, $\varphi\in\Lb^1(\Tn)$ whenever $\varphi\in\al^2(g)$.
Thus, for every $x\in\D^n$, we can define ${\sf eval}_{x}:\al^2(g)\lrarw \mathbb{C}$ by 
\[
{\sf eval}_{x}(\varphi)\,:=\,\poi[\varphi](x).
\]

It is routine to show that ${\sf eval}_x$ is a bounded linear functional for each
$x\in\D^n$. Hence, by the Riesz representation theorem, there exists a function in
$\al^2(g)$, which we shall denote in this paper by $\overline{K_{\al,\,g}}(x, \bcdot):\Tn\lrarw \C$, such that
\[
 {\sf eval}_{x}(\varphi)\,=\,\langle \varphi, \overline{K_{\al,\,g}}(x, \bcdot)\rangle_g\,.
\]
We call $K_{\al,\,g}(x, \bcdot)$ the {\em Szeg{\H{o}} kernel associated to $\al^2(g)$}.
\smallskip

\subsection{General functional analysis}\label{SS:gen_functanal}

In this subsection we state a couple of results that are perhaps not widely seen in the form
that we need (especially by readers who specialize in complex geometry or function theory).
The results themselves are very standard, and we shall only write a line or two about 
their proofs. For the first such result, we first recall: if $X$ is a Banach space, $S$ is a subspace of 
$X$ and $L$ is a subspace of $X^{*}$, then 
\begin{align*}
\rprp{S}\,&:=\,\{\lambda\in X^{*}:\lambda(x)=0 \ \ \forall x\in S\},\\
\lprp{L}\,&:=\,\{x\in X:\lambda(x)=0 \ \ \forall \lambda\in L\}.
\end{align*}

\begin{lemma}\label{L:isom1}
Let $X$ be a Banach space and $E$, $S$ be closed subspaces of $X$ with $E\subseteq S$. Let
$q:S\lrarw S/E$ be the quotient map. For each $F\in (S/E)^{*}$, the map 
\[
\Theta\,:\,F\longmapsto \widetilde{F\circ q}+\rprp{S},
\]
where $\widetilde{F\circ q}$ is any (fixed) norm-preserving $\C$-linear extension of $F\circ q$ to $X$,
is well defined and is an isometric isomorphism from $(S/E)^{*}$ to ${\rprp{E}}/{\rprp{S}}$.
\end{lemma}

The proof is utterly standard and runs along the lines of,
for instance, \cite[Theorem~4.9]{rudin:Func-Ana91}.
\smallskip

The second result of this subsection is about the dual of the space of trace class operators $\mathscr{T}(H)$,
where $H$ and $\mathcal{B}(H)$ are as in Section~\ref{S:intro}. Our presentation will be very brief,
and the reader is referred to \cite[Chapter~3, \S18]{conway:OpTh99}
for details of the concepts discussed below.
\smallskip

Given $T\in\mathcal{B}(H)$, write $|T| := (T^{*}T)^{\frac{1}{2}}$. If we fix an orthonormal basis
$\{e_j:j\in\mathbb{N}\}$ of $H$, the quantity
\begin{equation}
\sum_{j\in\mathbb{N}}\langle|T|e_j,\,e_j\rangle
\label{E:def_tr}
\end{equation}
is independent of the choice of the orthonormal basis $\{e_j:j\in\mathbb{N}\}$.
The space of {\em trace class operators}, denoted by $\mathscr{T}(H)$,
consists of operators $T\in\mathcal{B}(H)$ for which the quantity in \eqref{E:def_tr} is finite.
Thus, for a fixed $T\in\mathscr{T}(H)$, we have a number
\begin{equation}\label{E:tr_norm}
\|T\|_{\text{tr}}\,:=\,\sum_{j\in\mathbb{N}}\langle|T|e_j,\,e_j\rangle
\end{equation}
(where $\{e_j:j\in\mathbb{N}\}$ is some orthonormal basis). It is a fact that \eqref{E:tr_norm} defines
a norm and that $\mathscr{T}(H)$ is a Banach space with this norm.
\smallskip

We will need the concept of the trace of an operator in $\mathcal{B}(H)$. One fixes some orthonormal
basis on $H$ and attempts a definition as one would for a finite-dimensional $H$.
Convergence and independence of the choice of orthonormal basis hold true for any $T\in \mathscr{T}(H)$.
For any such $T$, we denote the trace by  ${\sf trace}(T)$. We will not spell out an expression for 
${\sf trace}(T)$\,---\,we refer the reader to  \cite[Chapter 3, \S18]{conway:OpTh99}. What follows from
the above procedure is that
\begin{equation}\label{E:trace_trnorm}
|\mathsf{trace}(T)|\,\leq\, \|T\|_{\text{tr}} \; \; \; \forall T\in \mathscr{T}(H).
\end{equation}

It turns out that $\mathscr{T}(H)$ is a two-sided ideal of $\mathcal{B}(H)$.
Moreover, given $T\in \mathcal{B}(H)$ and $A\in \mathscr{T}(H)$ we have:
\begin{equation}\label{E:prod_trnorm}
\|TA\|_{\text{tr}}\,\leq\,\|T\|_{\text{op}}\|A\|_{\text{tr}} \quad\text{and}
\quad \|AT\|_{\text{tr}}\,\leq\,\|T\|_{\text{op}}\|A\|_{\text{tr}},
\end{equation} 
where $\|T\|_{\text{op}}$ represents the operator norm of $T$. 
Because of the inequalities above, each $T\in\mathcal{B}(H)$ induces a linear functional
$L_T\in(\mathscr{T}(H))^{*}$ defined by $L_T(A) := \mathsf{trace}(TA)$.

\begin{result}\label{R:Dual_traceclass}
The map $\Lambda:\mathcal{B}(H)\lrarw (\mathscr{T}(H))^{*}$ defined by
\[
\Lambda(T)\,:=\,L_{T} \; \; \; \forall T\in \mathcal{B}(H),
\]
where $L_{T}$ is defined by $L_{T}(A) := \mathsf{trace}(TA) \ \forall A\in\mathscr{T}(H)$, gives an
isometric isomorphism of $\mathcal{B}(H)$ onto $(\mathscr{T}(H))^{*}$.
\end{result}

The above is a standard result; see, for instance, \cite[Theorem 19.2]{conway:OpTh99}.
\smallskip 

We end this subsection by reminding ourselves of rank-one operators, which will be
one of the tools for establishing a key result of the next section. Given $x$, $y\in H$,
we define the rank-one operator $x\otimes y$ as
\[ 
x\otimes y(v)\,:=\,\langle v, y\rangle x \; \; \; \forall v\in H.
\]
It is not hard to see that
$x\otimes y\in\mathscr{T}(H)$. Also, we have
\begin{equation}\label{E:trace}
\|x\otimes y\|_{\text{tr}}\,=\,\|x\|_{H} \|y\|_{H} \quad \text{and}
\quad {\sf trace}(x\otimes y)\,=\,\langle x, y\rangle.
\end{equation}
\smallskip

\section{Closure in the weak${}^{{\boldsymbol{*}}}$ 
topology on $\hinf$}\label{SS:charac_weak}

This section is devoted to providing a simple description of the $\wes$ closure of a uniform
subalgebra $\al\subseteq A(\D^n)$ (and more). Results of this kind are not entirely straightforward. 
McCullough presents results of this nature for a {\em certain class} of uniform algebras on general compact
Hausdorff spaces in \cite[Section~5]{mccullough:NPtida96}.
Using quite different methods, based on the Krein--{\v{S}}mulian theorem, we describe the
$\wes$ closure of any linear subspace of $A(\D^n)$. 
\smallskip

To be precise about the meaning of ``$\wes$'' here: with $H$ as in Section~\ref{S:funcAnal},
$\mathcal{B}(H)$ is endowed with the $\wes$ topology
that it acquires as the dual space of $\mathscr{T}(H)$\,---\,which follows from Result~\ref{R:Dual_traceclass}.
We recall that a net
$\{T_{\alpha}:\alpha\in J\}$, $J$ being a directed set, in $\mathcal{B}(H)$ converges to
$T\in\mathcal{B}(H)$ in the $\wes$ topology if and only if $\{{\sf trace}(T_{\alpha}A) : \alpha\in J\}$
converges to ${\sf trace}(TA)$ for every $A\in \mathscr{T}(H)$.
\smallskip

At this juncture, we shall fix our Hilbert space $H$ to be $H^2(\Tn)$.
Each $\varphi\in \hinf$ defines a {\em multiplier operator} $M_{\varphi}\in \mathcal{B}(H^2(\Tn))$ as follows.
It follows from a result of Marcinkiewicz and Zygmund on multiple
Poisson integrals that for any bounded function $u$
on $\D^n$, $n\geq 2$, that is harmonic in each variable separately, the limit
\begin{equation}\label{E:Fatou_style}
 \lim_{r\to 1^-}u(r\zt)\,=:\,\bv{u}(\zt), \; \; \zt\in \Tn, \; \text{exists for $m$-a.e. $\zt\in \Tn$};
\end{equation}
see \cite[Section~2.3]{rudin:ftp69}. When $n = 1$, the latter statement is the classical theorem of
Fatou. Furthermore, $\bv{u}$ is of class
$\Lb^\infty(\Tn)$, $n\geq 1$, and satisfies
\begin{equation}\label{E:poi_reproduces}
 u\,=\,\poi[\bv{u}] \quad\text{and}
 \quad \|\bv{u}\|_{\Lb^\infty(\Tn)}\,=\,\sup\nolimits_{\D^n}|u|.
\end{equation}
Since any holomorphic function on $\D^n$ is harmonic in each variable separately, it follows that to each
$\varphi\in \hinf$ is associated the radial boundary-value function $\bv{\varphi}$, which establishes
an isometry of $\hinf$ into $\Lb^\infty(\Tn)$. With these facts, we have
\begin{equation}\label{E:opnorm_mult}
 M_{\varphi}(h)\,:=\,\bv{\varphi}h \; \; \; \forall h\in H^2(\Tn) \; \quad\text{and}
 \quad \; \|M_{\varphi}\|_{\text{op}}\,=\,\sup\nolimits_{\D^n}|\varphi|.
\end{equation}

Identifying $\varphi$ with $M_{\varphi}$, we see
that $\hinf\hookrightarrow \mathcal{B}(H^2(\Tn))$.
In the $\wes$ topology on $\mathcal{B}(H^2(\Tn))$ viewed as the dual space of 
$\mathscr{T}(H^2(\Tn))$, $\hinf$ is $\wes$ closed\,---\,see \cite[Lemma~3.6]{mccullough:NPtida96}. 
We would like to understand better the $\wes$ topology restricted to $\hinf$.
\smallskip

In view of the above discussion, when we ascribe to subsets of
$\hinf$ properties of the $\wes$ topology, it will be understood that the discussion is
about the image of those subsets under the embedding $\hinf\hookrightarrow \mathcal{B}(H^2(\Tn))$.
\smallskip

In what follows, we shall abbreviate the inner product $\langle\bcdot, \bcdot\rangle_1$\,---\,see 
our notation in subsection~\ref{SS:intro_szeg_ker}\,---\,simply to $\langle\bcdot, \bcdot\rangle$. Similarly,
the classical Szeg{\H{o}} kernel: i.e., the Szeg{\H{o}} kernel associated to
$H^2(\Tn)$\,---\,which would be $K_{A(\D^n),\,1}(x, \bcdot)$ in the notation of
subsection~\ref{SS:intro_szeg_ker}\,---\,will be denoted by $K(x, \bcdot)$, $x\in \D^n$.

\begin{lemma}\label{L:Weaknet_pointwise}
Let $\{\varphi_{\alpha}\}_{\alpha\in J}$ be a weak${}^{{\boldsymbol{*}}}$ convergent net in $\hinf$.
Then $\exists\varphi\in\hinf$ such that $\{\varphi_{\alpha}(x)\}_{\alpha\in J}$ converges to $\varphi(x)$
for every $x\in\D^n$.
\end{lemma}
\begin{proof}
As $\hinf$ is $\wes$ closed and $\{\varphi_{\alpha}\}_{\alpha\in J}$ is $\wes$
convergent, $\exists\varphi\in \hinf$ such that $\{\varphi_\alpha\}_{\alpha\in J}$ converges to 
$\varphi$ in the $\wes$ topology. This implies
that $\{{\sf trace}(M_{\varphi_\alpha}A)\}_{\alpha\in J}$ converges to ${\sf trace}(M_{\varphi}A)$ for every 
$A\in \mathscr{T}(H^2(\Tn))$. Take $A = 1\otimes K(x, \bcdot)$.
Then, by \eqref{E:trace}:
\begin{align}
\mathsf{trace}(M_{\varphi_\alpha}[1\otimes K
(x, \bcdot)])\,&=\,\mathsf{trace}(\varphi_\alpha\otimes K(x, \bcdot)) \label{E:tr_interim} \\
&=\,\langle\varphi_\alpha,K(x, \bcdot)\rangle \notag \\
&=\,\varphi_\alpha(x). \notag
\end{align}
Similarly $\mathsf{trace}(M_{\varphi} [1\otimes K(x, \bcdot)])=\varphi(x)$. By the
discussion preceding \eqref{E:tr_interim}, the lemma follows.
\end{proof}

Our next result gives a characterization of $\wes$ convergent sequences in $\hinf$. Before
we present it we note: by the fact that
\begin{itemize}
 \item the class of all finite-rank operators in $\mathscr{T}(H)$, and
 \item the set $\{K(x, \bcdot) : x\in\D^n\}$ in $H^2(\Tn)$
\end{itemize}
are dense in their respective norms, it follows that the finite-rank operators of the form
$\sum_{j=1}^{M}f_j\otimes K(\pj, \bcdot)$,
where $f_j\in H^2(\Tn)$ and $\pj\in\D^n$, are dense in $\mathscr{T}(H^2(\Tn))$.
\smallskip

Just the ``only if'' implication of the following proposition is needed to establish Proposition~\ref{P:closure_wes}.
However, we present a characterization of $\wes$ convergent sequences in $\hinf$, as it may be of
independent interest.

\begin{proposition}\label{P:Weakseq_charac}
Let $\{\varphi_{\nu}\}_{\nu\in\mathbb{N}}$ be a sequence in $\hinf$. Then
$\{\varphi_{\nu}\}_{\nu\in\mathbb{N}}$ is weak${}^{{\boldsymbol{*}}}$ convergent if and only if
\begin{itemize}
\item[$(i)$]$\sup\{\sup_{\D^n}|\varphi_{\nu}| : \nu\in\mathbb{N}\} < \infty$; and
\item[$(ii)$]$\{\varphi_{\nu}(x)\}_{\nu\in\mathbb{N}}$ converges to $\varphi(x)$
for some $\varphi\in\hinf$ and for each $x\in\D^n$.
\end{itemize}
\end{proposition}
\begin{proof}
Let $\{\varphi_\nu\}_{\nu\in\mathbb{N}}$ be a $\wes$ convergent sequence.
Then $(ii)$ follows from Lemma~\ref{L:Weaknet_pointwise}. To obtain $(i)$, 
we consider the linear functionals
$L_{M_{\varphi_\nu}}\in (\mathscr{T}(H^2(\Tn)))^{*}$: see Result~\ref{R:Dual_traceclass}.
For a fixed $A\in\mathscr{T}(H^2(\Tn))$, 
the convergence of  $\{\varphi_{\nu}\}$\,---\,hence of $\{{\sf trace}(M_{\varphi_{\nu}}A)\}$\,---\,implies
that the sequence $\{|L_{M_{\varphi_\nu}}(A)|\}$ is bounded.
Hence, it follows from the Uniform Boundedness Principle that $\sup\|L_{M_{\varphi_\nu}}\|_{\text{op}} < \infty$.
But, by Result~\ref{R:Dual_traceclass} and \eqref{E:opnorm_mult}:
\[
 \|L_{M_{\varphi_\nu}}\|_{\text{op}}\,=\,\|M_{\varphi_\nu}\|_{\text{op}}\,=\,\sup\nolimits_{\D^n}|\varphi_\nu|,
\]
whence we have $(i)$.
\smallskip

To establish the converse, we shall use the density of the finite-rank operators 
of the form $\sum_{j=1}^{M}f_j\otimes K(\pj, \bcdot)$ in $\mathscr{T}(H^2(\Tn))$.
So let us consider a sequence $\{\varphi_{\nu}\}\subset \hinf$ for which $(i)$ and $(ii)$ are satisfied.
Let $A:=\sum_{j=1}^{M}f_j\otimes K(\pj, \bcdot)$, $f_j\in H^2(\Tn)$. Then we have:
\begin{align*}
 {\sf trace}(M_{\varphi_\nu}A)\,&=\,{\sf trace}\bigg(\sum_{j=1}^{M}
  \varphi_{\nu}f_j\otimes K(\pj, \bcdot)\bigg)\\
&=\,\sum_{j=1}^{M}\big\langle\varphi_{\nu}f_j, K(\pj, \bcdot)\big\rangle\\
&=\,\sum_{j=1}^{M}\varphi_{\nu}(\pj)f_{j}(\pj).
\end{align*}
By $(ii)$ and the above calculation, it follows that
${\sf trace}(M_{\varphi_\nu}A)$ converges to ${\sf trace}(M_\varphi A)$ for each $A$ having the form 
$\sum_{j=1}^{M}f_j\otimes K(\pj, \bcdot)$, where $f_j\in H^2(\Tn)$ and $\pj\in\D^n$.
\smallskip

Now let $A\in\mathscr{T}(H^2(\Tn))$ be an arbitrary element. Given $\eps>0$ there exist functions
$f_j\in H^2(\Tn)$ and points
$\pj\in\D^n$, $1\leq j\leq M$ such that
\begin{equation}\label{E:eps/3}
 \big\|A-\sum\nolimits_{j=1}^{M}f_j\otimes K(\pj, \bcdot)\big\|_{\text{tr}}\,<\,
 \frac{\eps}{3\big(\sup\nolimits_{\nu\in \nat}\{\sup\nolimits_{\D^n}|\varphi_{\nu}|\}\big)}.
\end{equation}
That the right-hand side above is well defined follows from our assumption $(i)$.
We compute:
\begin{align*}
 |{\sf trace}(M_{\varphi_\nu}A) -
  {\sf trace}(M_{\varphi}A)|\,&\leq\,\Big|{\sf trace}\Big(M_{\varphi_\nu}\Big(
  A - \sum_{j=1}^{M}f_j\otimes K(\pj, \bcdot)\Big)\Big)\Big| \\
 &\quad+\Big|{\sf trace}\Big((M_{\varphi_\nu} - M_{\varphi})\Big(\sum_{j=1}^{M}f_j\otimes K(\pj, \bcdot)\Big)
  \Big)\Big| \\
 &\quad+\Big|{\sf trace}\Big(M_{\varphi}\Big(A - \sum_{j=1}^{M}f_j\otimes K(\pj, \bcdot)\Big)\Big)\Big|.
\end{align*}
Observe that, by \eqref{E:trace_trnorm}, \eqref{E:prod_trnorm}, \eqref{E:opnorm_mult}
and \eqref{E:eps/3}, the first and third terms above are dominated by
$\eps/3$ irrespective of $\nu$. It is now easy to see from the discussion in the preceding paragraph that  
there exists $N_{0}\in \Z_+$ such that the second term is dominated by $\eps/3$ for every $\nu\geq N_{0}$.
Putting these together, we get that ${\sf trace}(M_{\varphi_\nu}A)\!\lrarw {\sf trace}(M_{\varphi}A)$ as
$\nu\to \infty$ for each fixed $A\in\mathscr{T}(H^2(\Tn))$. 
Hence $\varphi_\nu$ converges to $\varphi$ in the $\wes$ topology, and we are done.
\end{proof}

The significance of Proposition~\ref{P:Weakseq_charac} to the goal of this section
is made clear by Lemma~\ref{L:clos_equiv_seqClos} below. Before we state this lemma, we
need the following result.

\begin{result}\label{R:Krein_Smu}
Let $X$ be a separable Banach space and $C$ is a convex subset of $X^{*}$. Then $C$ is weak${}^{{\boldsymbol{*}}}$
closed if and only if it is sequentially weak${}^{{\boldsymbol{*}}}$ closed.
\end{result}

This result is a consequence of the Krein--{\v{S}}mulian Theorem\,---\,see, for instance,
\cite[Chapter 5, \S12]{conway:FuncAna90} and Corollary~12.7 therein.

\begin{lemma}\label{L:clos_equiv_seqClos}
A subspace of $\hinf$ is weak${}^{{\boldsymbol{*}}}$ closed if and only if it is sequentially
weak${}^{{\boldsymbol{*}}}$ closed.
\end{lemma}
\begin{proof}
Recall that the reference to a subspace of $\hinf$ in the $\wes$ topology alludes to its isometric embedding
into $\mathcal{B}(H^2(\Tn))$. Fix a subspace $S\subseteq \hinf$. Next, set
\[
 X\,=\,\mathscr{T}(H^2(\Tn)) \quad \text{and}\quad C\,=\,{\sf j}(S),
\]
where ${\sf j}$ is the linear isometric embedding of $\hinf$ discussed above. As
$\mathscr{T}(H^2(\Tn))$ is separable, the lemma follows from Result~\ref{R:Krein_Smu}.
\end{proof}

The above lemma gives us the main result of this section.

\begin{proposition}\label{P:closure_subalg}
Let $S$ be a $\C$-linear subspace of $A(\D^n)$. Then:
\begin{itemize}
 \item[$(1)$] the closure of $S$ in the weak${}^{{\boldsymbol{*}}}$ topology equals
 \[
  (\text{the closure of $\left.S\right|_{\D^n}$ in the topology of pointwise convergence})\cap \Lb^\infty(\D^n).
 \]
 \item[$(2)$] the closure of $S$ in the weak${}^{{\boldsymbol{*}}}$ topology equals
 \[
  (\text{the closure of $\left.S\right|_{\D^n}$ in the topology of local uniform convergence})\cap \Lb^\infty(\D^n).
  \]
\end{itemize}
In particular, the closure of a uniform subalgebra $\al\subseteq A(\D^n)$ is $\dal(\al)$. 
\end{proposition}
\begin{proof}
The proof of $(1)$ is immediate from the last lemma and Proposition~\ref{P:Weakseq_charac}. Now,
given an element $\varphi$ in the $\wes$ closure of $S$,
any $\wes$ convergent sequence
$\{\varphi_\nu\}\subset S$ of which $\varphi$ is the
pointwise limit is\,---\,by 
Proposition~\ref{P:Weakseq_charac}\,---\,uniformly bounded.
By Montel's Theorem and the 
pointwise convergence of the latter sequence, we deduce that $\varphi_\nu\!\lrarw \varphi$ locally uniformly.
Hence $(2)$ follows.
\end{proof}
\medskip

\section{Some function theory in several complex variables}\label{S:funcThr}

Although we have used the term ``uniform algebra'' several times above, it might be helpful to recall the
definition. Given a compact Hausdorff space $X$, a {\em uniform algebra on $X$} is a subalgebra
of $\smoo(X; \C)$ that is closed with respect to the uniform norm, contains the constants, and separates
the points of $X$. Given a uniform algebra $A$, we call a subalgebra $B\subset A$ a {\em uniform subalgebra of $A$}
if $B$ is itself a uniform algebra.
\smallskip

In this paper, we are interested in uniform algebras on $\overline{\D^n}$. We begin with the following result.

\begin{lemma}\label{L:0-1lemma}
Let $X_1,\dots, X_N$, $N\geq 2$, be distinct points in $\D^n$. Let $\al$ be a uniform subalgebra of $A(\D^n)$. There
exist functions $\Phi_1,\dots, \Phi_N \in \al$ such that
\[
 \Phi_j(X_k)\,=\,\Kdel_{jk}, \; \; j, k = 1,\dots, N,
\]
where $\Kdel_{jk}$ denotes the Kronecker symbol.
\end{lemma}

The proof of this lemma relies on the fact that $\al$ separates points on $\D^n$ and is closed under multiplication.
We shall skip the proof since it is utterly elementary.
\smallskip

The above lemma is essential to Proposition~\ref{P:SzegoRep}, which we shall use several times in
Sections~\ref{S:interpCond} and \ref{S:non-mod1Set_proof}. First, we need some notations.
Let $X_1,\dots, X_N$ be as in Lemma~\ref{L:0-1lemma} and fix a uniform subalgebra $\al\subseteq A(\D^n)$.
Denote the set $\{X_1,\dots, X_N\}$ by $\boldsymbol{{\sf X}}$, and write
\[
 \ide{\al}\,:=\,\text{the $\wes$ closure of $\IX{\al}$ (viz.,
 the ideal of all functions in $\al$ that vanish on $\boldsymbol{{\sf X}}$).}
\]
Note that, by Proposition~\ref{P:closure_subalg}, each $\psi\in \ide{\al}$ is a bounded holomorphic function. Thus,
by the discussion at the beginning of Section~\ref{SS:charac_weak}, the following make sense:
\begin{align*}
 \lprp{\!\ide{\al}}\,&:=\,\left\{f\in \Lb^1(\Tn) : \inttor\!\bv{\psi}f dm = 0 \ \text{for each} \ \psi\in \ide{\al}\right\}, \\
 \lprp{\dal(\al)}\,&:=\,\left\{f\in \Lb^1(\Tn) : \inttor\!\bv{\psi}f dm = 0 \ \text{for each} \ \psi\in \dal(\al)\right\}.
\end{align*}

We recall our abbreviated notation: $K(x, \bcdot)$, $x\in \D^n$, denotes the classical
Szeg{\H{o}} kernel: i.e., the Szeg{\H{o}} kernel associated to $H^2(\Tn)$.
With these notations, we state:

\begin{proposition}\label{P:SzegoRep}
Let $X_1,\dots, X_N$, $N\geq 2$, be distinct points in $\D^n$. Let $\al$ be a uniform subalgebra of $A(\D^n)$.
For each $f\in \lprp{\!\ide{\al}}$, let $[f]$ denote the $\lprp{\dal(\al)}$-coset of $f$. There exist constants
$a_1,\dots, a_N\in \C$, which are independent of the choice of representative of the
coset $[f]\in \lprp{\!\ide{\al}}/\lprp{\dal(\al)}$, such that
\[
 [f]\,=\,\bigg[\sum_{1\leq j\leq N} a_jK(X_j, \bcdot)\bigg].
\]
\end{proposition}
\begin{proof}
Let us define a linear functional $\fL_{[f]} : \dal(\al)\lrarw \C$ by
\begin{equation}\label{E:L-map1}
 \fL_{[f]}(\phi)\,:=\,\inttor\!\bv{\phi}f dm.
\end{equation}
We must first establish the following:
\vspace{1.5mm}

\noindent{{\bf Claim.} {\em $\fL_{[f]}$ is independent of the choice of
representative of the coset $[f]\in \lprp{\!\ide{\al}}/\lprp{\dal(\al)}$.}}
\vspace{0.5mm}

\noindent{Suppose 
$\widetilde{f}$ is some other representative of the coset $[f]$. Then, there exists
a $g\in \lprp{\dal(\al)}$ such that $\widetilde{f} = f+g$.
By the definition of $\lprp{\dal(\al)}$, we have:
\[
 \inttor\!\bv{\phi}\widetilde{f} dm\,=\,\inttor\!\bv{\phi}f dm + \inttor\!\bv{\phi}g\,dm\,=\,\inttor\!\bv{\phi}f dm.
\]
Since $\phi$ was chosen arbitrarily from $\dal(\al)$, the claim follows.
\smallskip

Since $f\in \lprp{\!\ide{\al}}$, $\fL_{[f]}$ vanishes on $\ide{\al}$.}
\smallskip

By Lemma~\ref{L:0-1lemma}, we can find functions $\Phi_1,\dots, \Phi_N\in \al$ such that
\[
 \Phi_j(X_k)\,=\,\Kdel_{jk}, \; \; j, k = 1,\dots, N.
\]
Set $a_j := \fL_{[f]}(\Phi_j)$, $j = 1,\dots, N$. For each $\phi\in \dal(\al)$, write
\[
 \widetilde{\phi}\,:=\,\phi -\!\!\sum_{1\leq j\leq N}\phi(X_j)\Phi_j,
\]
which belongs to $\ide{\al}$ (since the $\wes$ closed ideals $\ide{\al}$ and
$\{\psi\in \dal(\al) : \psi(x) = 0 \ \forall x\in \boldsymbol{{\sf X}}\}$ coincide). Thus
\begin{align}
 \fL_{[f]}(\phi)\,&=\,\fL_{[f]}\bigg(\sum_{1\leq j\leq N}\phi(X_j)\Phi_j\bigg) \notag \\
 	&=\,\sum_{1\leq j\leq N}a_j\phi(X_j) \notag \\
 	&=\,\sum_{1\leq j\leq N}a_j\int\nolimits_{\Tn}\!\bv{\phi}K(X_j,\bcdot)dm \; \; \; \forall \phi\in \dal({\al}).
 	\label{E:int_with_kernels}
\end{align}
In the last equality, we use the fact that $\bv{\phi}$ is the boundary-value function of a function in $\hinf$
and, therefore, is in $H^2(\Tn)$. Then, \eqref{E:int_with_kernels} follows from the discussion
in subsection~\ref{SS:intro_szeg_ker}. But note that the function
\[
 \sum_{1\leq j\leq N} a_jK(X_j, \bcdot) \in \Lb^1(\Tn)
\]
itself belongs to $\lprp{\!\ide{\al}}$. Thus, from \eqref{E:int_with_kernels}, we see that $f$ and 
$\sum_{1\leq j\leq N} a_jK(X_j, \bcdot)$ differ by a function in $\lprp{\dal(\al)}$. Hence the result.
\end{proof}

The final result of this section is central to the proof of Theorem~\ref{T:interp-Char}. At its heart is a close
reading of the reason for the well-known fact that $\left.A(\D^n)\right|_{\Tn}$ is approximating in modulus
(see the paragraph following Result~\ref{R:McC} for a definition). The class $\polcl(R)$ below is as
defined in the statement of Theorem~\ref{T:interp-Char}.

\begin{proposition}\label{P:posFn_poly}
Fix a positive integer $R\geq 1$. Let $f$ be a positive, continuous function on $\Tn$. For each $\eps > 0$, there exists a
polynomial $p\in \polcl(R)$ such that
\[
 \sup\nolimits_{\,\Tn}|f - |p|^2|\,<\,\eps.
\] 
\end{proposition}
\begin{proof}
Let $\Fej_k$ denote the $k$-th Fej{\'e}r kernel on $\Tn$ (i.e., the kernel associated to the Ces{\`a}ro mean
involving the characters parametrized by $(\alpha_1,\dots, \alpha_n)\in \Z^n$, $-k\leq \alpha_j\leq k$). Since
$f$ is positive and continuous, $\log(f)$ is continuous as well. By Fej{\'e}r's theorem:
\begin{equation}\label{E:Cesaro}
 \log(f)*\Fej_k\!\lrarw \log{f} \; \; \text{uniformly, as $k\to \infty$}.
\end{equation}
By the properties of the Fej{\'e}r kernels, $\log(f)*\Fej_k$ is a trigonometric polynomial and, as $\log(f)$
is real-valued, there exist polynomials $P_k\in \C[z_1,\dots, z_n]$ such that
\[
 \log(f)*\Fej_k(\cis{\tht_1},\dots, \cis{\tht_n})\,=\,\re\big(P_k(\cis{\tht_1},\dots, \cis{\tht_n})\big).
\]
Let us now define $g_k : \C^n\lrarw \C$ by
$g_k(z)\,:=\,e^{P_k(z_1,\dots, z_n)/2}$, $z\in \C^n$.
By \eqref{E:Cesaro} and the fact that $|e^A| = e^{\re(A)}$ for any $A\in \C$, we get
\begin{equation}\label{E:g_approx}
 \big|\left. g_k\right|_{\Tn}\big|^2\lrarw f \; \; \text{uniformly, as $k\to \infty$}.
\end{equation}

Let us now set
\[
 m\,:=\,\max\nolimits_{\zt\in \Tn}f(\zt), \quad\text{and} \quad 
 M\,:=\,\sqrt{2}\,\sqrt{(m+ \newfrc{\eps}{2}) + ((m+ \newfrc{\eps}{2})^{1/2} + 1)^2}.
\]
For simplicity of notation, let us abbreviate $\sup_{\,\Tn}|\bcdot|$ to $\|\bcdot\|_{\Tn}$.
By \eqref{E:g_approx}, there exists a positive integer $k^\eps$ such that
\begin{equation}\label{E:approx_1stHalf}
 \big\|\,|g_k|^2 - f\big\|_{\Tn}\,<\,\eps/2 \; \; \; \forall k\geq k^\eps.
\end{equation}
Now set
\[
 \mu_{R,\,\eps}\,:=\,\min\nolimits_{R\,\bcdot\,\overline{\D^n}}|g_{k^\eps}| \; \; \; (\text{which is a
 strictly positive number}).
\]
The Taylor expansion of $g_{k^\eps}$, the latter being entire, converges to $g_{k^\eps}$ uniformly on
any fixed compact subset of $\Cn$. Thus, we can find a polynomial $p\in \C[z_1,\dots, z_n]$ such that
\begin{equation}\label{E:Taylor_trunc}
 \sup\nolimits_{\,R\,\bcdot\,\overline{\D^n}}|g_{k^\eps} - p|\,<\,\min
 \left(\frac{\eps}{2M}, \frac{\mu_{R,\,\eps}}{2}, 1\right).
\end{equation}
By our definition of $\mu_{R,\,\eps}$, $p^{-1}\{0\}\cap (R\bcdot\overline{\D^n}) = \varnothing$. Hence,
$p\in \polcl(R)$.
\smallskip

Finally\,---\,making use of \eqref{E:Taylor_trunc}\,---\,we estimate:
\begin{align*}
 \|\,|g_{k^\eps}|^2 - |p|^2\|_{\Tn}\,&\leq\,\|g_{k^\eps} + p\|_{\Tn}\times \|g_{k^\eps} - p\|_{\Tn} \\
 	&\,\leq\,\sqrt{2}\,\sqrt{\|g_{k^\eps}\|^2_{\Tn} + \|p\|^2_{\Tn}} \ \frac{\eps}{2M} \\
 	&\,\leq\,\eps/2.
\end{align*}
By the above estimate and \eqref{E:approx_1stHalf}, we see that $p$ is the desired polynomial.
\end{proof}

\section{The proof of Theorem~\ref{T:interp-Char}}\label{S:interpCond}

Before we give a proof of Theorem~\ref{T:interp-Char}, it will be very useful to state a
special case of  Lemma~\ref{L:isom1} adapted to the situation that is of interest to us. The
spaces of greatest interest to us are the quotient spaces:
\begin{equation}\label{E:quotients}
 \dal(\al)/\ide{\al} \quad \text{and} \quad \lprp{\!\ide{\al}}/\lprp{\dal(\al)},
\end{equation}
these spaces being exactly as introduced in Section~\ref{S:funcThr}. 
Since this lemma will require some preliminary discussion, we divide this section into two subsections.
\smallskip

\subsection{A few essential auxiliary lemmas}\label{SS:essential}

We will need to work with a more general collection of objects than $\al$.
To this end\,---\,with $\al\subseteq A(\D^n)$
as above\,---\,let $\I$ denote a {\em uniformly closed} ideal of $\al$. Write
\[
 \Ide\,:=\,\text{the $\wes$ closure (in the sense of Section~\ref{SS:charac_weak}) of $\I$}.
\]
As $\I$ is a subspace of $\al\subset \hinf$ we can, in view of Proposition~\ref{P:closure_subalg}
and the discussion at the beginning of Section~\ref{SS:charac_weak}, define:
\[ 
 \lprp{\!\Ide}\,:=\,\Big\{f\in \Lb^1(\Tn) : \inttor\!\bv{\psi}f dm = 0 \ \text{for each} \ \psi\in \Ide\Big\}.
\]
With these notations, we have:

\begin{lemma}\label{L:lprps}
Let $\al$ be a uniform subalgebra of $A(\D^n)$, and let $\I$ be a uniformly closed
ideal in $\al$. Then
\[
 \lprp{\!\Ide}\,=\,\lprp{\I}\,:=\,\Big\{f\in \Lb^1(\Tn) :
 						\inttor\!\bv{\psi}f dm = 0 \ \text{for each} \ \psi\in \I\Big\}.
\]
\end{lemma}
\begin{proof}
It is clear that $\lprp{\!\Ide}\subseteq \lprp{\I}$.
Consider an arbitrary $\phi\in \Ide$. By \eqref{E:Fatou_style} (we reiterate: owing to
Proposition~\ref{P:closure_subalg}, $\phi\in \hinf$), we have:
\[
 \left.\phi(r\bcdot)\right|_{\Tn}\lrarw \bv{\phi} \; \; \text{$m$-a.e. as $r\to 1^-$.}
\]
Invoking Proposition~\ref{P:closure_subalg} once more, there exists a sequence
$\{\varphi_{\nu}\} \subset \I$ such that $\varphi_{\nu}\!\lrarw \phi$ uniformly on compact
subsets of $\D^n$. Let us fix an $r\in (0, 1)$. Then:
\[
 \varphi_{\nu}(r\bcdot)\!\lrarw \phi(r\bcdot) \; \; \text{uniformly on each $\overline{D(0; \rho)}\,^n$, 
 	$\rho\in (0, 1)$.}
\]
By Proposition~\ref{P:closure_subalg}, $\phi(r\bcdot)\in A(\D^n)\cap \Ide$, and hence in 
$\I$, for every $r\in (0, 1)$. 
Since $\bv{\phi}\in \Lb^\infty(\Tn)$, we may apply the dominated convergence theorem to get:
\begin{equation}\label{E:zero_clos}
 0\,=\,\lim_{r\to 1^-}\inttor\!\big(\left.\phi(r\bcdot)\right|_{\Tn}\big)g\,dm\,=\,\inttor\!\bv{\phi}g\,dm \; \; \;
 \forall g\in \lprp{\I}.
\end{equation}
This establishes that $\lprp{\I}\subseteq \lprp{\!\Ide}$, and hence the result.
\end{proof}

With $\I$ as above we shall write
\[
 \left.\I\right|_{\Tn}\,:=\,\{\left.\psi\right|_{\Tn} : \psi\in \I\},
\]
which is a subspace of $\Lb^\infty(\Tn)$.
The following lemma may seem a bit mysterious at the moment, but its need will become
clear in proving the principal lemma of this subsection.

\begin{lemma}[a part of {\cite[Proposition~5.9]{mccullough:NPtida96}}]\label{L:clos_new}
Let $\al$ and $\I\subseteq \al$ be as in Lemma~\ref{L:lprps}. Write
\begin{align*}
 \weak\,&:=\,\text{the closure of $\left.\I\right|_{\Tn}$ in the weak${}^{{\boldsymbol{*}}}$ topology on
 $\Lb^\infty(\Tn)$ as the dual of $\Lb^1(\Tn)$}, \\
 \I^2\,&:=\,\text{the closure of $\left.\I\right|_{\Tn}$ in $\Lb^2(\Tn)$}.
\end{align*}
Then, $\weak\subseteq \I^2\cap \Lb^\infty(\Tn)$.
\end{lemma}

The above forms the first one-third of the proof of \cite[Proposition~5.9]{mccullough:NPtida96}.
Apart from having to work with $\I$, there is
no difference
between the proof of Lemma~\ref{L:clos_new} and that in \cite{mccullough:NPtida96}. Therefore,
we shall not repeat McCullough's argument.
\smallskip

The principal lemma of this subsection is as follows. But first, a few more words on our notation: we
shall use $[\,\bcdot\,]$ to denote cosets in either of
the two quotient spaces named in \eqref{E:quotients}. However, we shall avoid ambiguity by using Greek
letters when referring to cosets in $\dal(\al)/\ide{\al}$ and standard Roman italics when referring to cosets in 
$\lprp{\!\ide{\al}}/\lprp{\dal(\al)}$.
 
\begin{lemma}\label{L:isom2}
Let $\al$ be a uniform subalgebra of $A(\D^n)$ and let $\boldsymbol{{\sf X}} = \{X_1,\dots, X_N\}$, where
the latter points are as in Theorem~\ref{T:interp-Char}.
For each $[\phi]\in \dal(\al)/\ide{\al}$, define
\[
 L_{[\phi]}([f])\,:=\,\int_{\Tn}\!\!\bv{\phi}f dm \quad \forall [f]\in \lprp{\!\ide{\al}}/\lprp{\dal(\al)}.
\] 
Then:
\begin{enumerate}
 \item[$(1)$]  $L_{[\phi]}([f])$ is independent of the choice of
 representatives of the cosets $[f]\in \lprp{\!\ide{\al}}/\lprp{\dal(\al)}$ and 
 $[\phi]\in \dal(\al)/\ide{\al}$. Furthermore, $L_{[\phi]}$ is an element of $(\lprp{\!\ide{\al}}/\lprp{\dal(\al)})^*$.
 \vspace{0.6mm}
 \item[$(2)$] $\|[\phi]\| = \|L_{[\phi]}\|_{{\rm op}}$ for every $[\phi]\in \dal(\al)/\ide{\al}$.
\end{enumerate}
\end{lemma}
\begin{proof}
The proof of $(1)$ is routine in view of the Claim in the proof of Proposition~\ref{P:SzegoRep}.
Note that $L_{[\phi]}([f]) = \fL_{[f]}(\phi)$ of Proposition~\ref{P:SzegoRep}.
Thus, we already have a proof of the independence of
$L_{[\phi]}([f])$ of the choice of the representative of the coset $[\phi]$.
\smallskip

The independence
of the choice of representative of the coset $[f]$ follows from the definition of  $\lprp{\dal(\al)}$. That
$L_{[\phi]}\in (\lprp{\!\ide{\al}}/\lprp{\dal(\al)})^*$ is now routine.
\smallskip

To prove $(2)$, we appeal to Lemma~\ref{L:isom1}. We take
\[
X\,=\,\Lb^1(\Tn), \quad S\,=\,\lprp{\!\ide{\al}}, \; \; \text{and} \; \; E\,=\,\lprp{\dal(\al)}
\]
to get
\begin{equation}\label{E:duality}
 \big(\lprp{\!\ide{\al}}/\lprp{\dal(\al)}\big)^*\,\cong_{{\rm isometric}}\,(\lprp{\dal(\al)})^\perp/(\lprp{\ide{\al}})^\perp.
\end{equation}
We now need to understand\,---\,in the notation of Lemma~\ref{L:isom1}\,---\,the coset $\Theta(L_{[\phi]})$.
However, this will actually require us to better understand the subspaces
$(\lprp{\!\ide{\al}})^\perp, (\lprp{\dal(\al)})^\perp \subset \Lb^{\infty}(\Tn)$. By Lemma~\ref{L:lprps},
\begin{equation}\label{E:lprps}
 \lprp{\!\ide{\al}}\,=\,\lprp{\IX{\al}}, \quad\text{and}
 \quad \lprp{\dal(\al)}\,=\,\lprp{\al}.
\end{equation}
Recall the definitions of
$\lprp{\!\ide{\al}}$ and $\lprp{\dal(\al)}$\,---\,it follows from \eqref{E:lprps} and the $\Lb^1$--$\Lb^\infty$ duality
that (see \cite[Theorem~4.7]{rudin:Func-Ana91}, for instance):
\begin{align}
 (\lprp{\!\ide{\al}})^\perp\,=\,&\text{the closure of $\left.\IX{\al}\right|_{\Tn}$
 							in the $\wes$ topology on $\Lb^\infty(\Tn)$} \notag \\
 						&\text{viewed as the dual of $\Lb^1(\Tn)$}, \label{E:rprp_Ideal} \\
 (\lprp{\dal(\al)})^\perp\,=\,&\text{the closure of $\left.\al\right|_{\Tn}$
 							in the $\wes$ topology on $\Lb^\infty(\Tn)$} \notag \\
 						&\text{viewed as the dual of $\Lb^1(\Tn)$}. \label{E:rprp_dualAlg}
\end{align}
Observe that each of the subspaces of $\Lb^\infty(\Tn)$ on the right-hand sides of the above equations are
of the form $\weak$ for an appropriate $\I$. 
\vspace{1.5mm}

\noindent{{\bf Claim.} $\weak = \bv{\Ide} := \{\bv{\psi} : \psi\in \Ide\}$.}
\vspace{0.5mm}

\noindent{By the $\Lb^1$--$\Lb^\infty$ duality that we have referred to above $\weak = (\lprp{\I})^\perp$. However,
by Lemma~\ref{L:lprps}, $\weak = (\lprp{\Ide})^\perp$. Hence, by the same duality principle
\begin{align*}
 \weak\,&=\,\text{the closure of $\bv{\Ide}$ in the $\wes$ topology on
 $\Lb^\infty(\Tn)$ as the dual of $\Lb^1(\Tn)$} \\
 &\supseteq\,\bv{\Ide}.
\end{align*}
Hence, it suffices to prove that $\weak \subseteq \bv{\Ide}$.
Pick a function $\psi\in \weak$. By Lemma~\ref{L:clos_new}, there exists a sequence
$\{\varphi_{\nu}\} \subset \I$ that converges to $\psi$ in $\Lb^2(\Tn)$-norm. But now, since
$\I\subset \I^2\subset H^2(\Tn)$,
\[
 \varphi_{\nu}(x) - \poi[\psi](x)\,=\,\int_{\Tn}\!\!\big(\varphi_{\nu}(\zt) - \psi(\zt)\big)K(x, \zt)\,dm(\zt),
 \; \; \; x\in \D^n.
\]
We recall that $K(x, \zt) = 1/\prod_{j=1}^n(1 - \overline{\zt}_j x_j)$ (recall that our reproducing kernels
are defined relative to the {\em normalized} measure $m$). Thus, from the above equation, we get
\begin{align*}
 |\varphi_{\nu}(x) - \poi[\psi](x)|\,&\leq\,\bigg[\int_{\Tn} \prod\nolimits_{j=1}^n|1-\overline{\zt}_j x_j|^{-2}\,dm(\zt)\bigg]^{1/2}
 							\big\|\left.\varphi_{\nu}\right|_{\Tn} - \psi\big\|_{\Lb^2(\Tn)} \\
 &\leq\,\frac{\big\|\left.\varphi_{\nu}\right|_{\Tn} - \psi\big\|_{\Lb^2(\Tn)}}
 			{\prod_{j=1}^n{\sf dist}(\pi_j({\sf C}), \bdy\D)} \quad \forall x\in {\sf C},
\end{align*}
where ${\sf C}$ is any compact subset of $\D^n$ and $\pi_j$ denotes the projection of $\Cn$ onto the $j$-th coordinate.
Thus, $\varphi_{\nu}\!\lrarw \poi[\psi]$ uniformly on compact subsets of $\D^n$. Now, recall that 
$\psi\in \Lb^\infty(\Tn)$. Thus, by Proposition~\ref{P:closure_subalg}, $\poi[\psi]\in \Ide$ and by
\eqref{E:poi_reproduces}, $\psi\in \bv{\Ide}$. Hence the claim.}
\smallskip

The above claim, together with \eqref{E:duality}, \eqref{E:rprp_Ideal} and \eqref{E:rprp_dualAlg}, gives us a very
useful identity:
\begin{equation}\label{E:duality_simpler}
 \big(\lprp{\!\ide{\al}}/\lprp{\dal(\al)}\big)^*\,\cong_{{\rm isometric}}\,\dal(\al)/\ide{\al},
\end{equation}
where the isometry is given by the isomorphism $\Theta$ described in Lemma~\ref{L:isom1}.
\smallskip

By \eqref{E:duality_simpler}, $\Theta(L_{[\phi]})$ is a coset in $\dal(\al)/\ide{\al}$, which we
shall call $[\theta_{[\phi]}]$. As $\Theta$ is an isometry,
\begin{equation}\label{E:isomet}
 \| [\theta_{[\phi]}] \|\,=\,\|L_{[\phi]}\|_{{\rm op}}\,.
\end{equation}
Unravelling the construction of $\Theta$ (and by the manner in which a function in $\Lb^\infty(\Tn)$
induces a bounded linear functional of $\Lb^1(\Tn)$) we have that for any $F\in (\lprp{\!\ide{\al}}/\lprp{\dal(\al)})^*$
\[
 F([f])\,=\,L_{\Theta(F)}([f]) \; \; \; \forall\,[f]\in \lprp{\!\ide{\al}}/\lprp{\dal(\al)}.
\]
Thus, if $\phi$ is any representative of $[\phi]$ and $\theta$ any representative of $[\theta_{[\phi]}]$, then:
\begin{align*}
 L_{[\phi]}([f])\,=\,\inttor\!\bv{\theta}\,\widetilde{g}\,dm \; \; \; &\forall\,\widetilde{g}\in [f] \; \; \text{and} \\
 &\forall\,[f]\in \lprp{\!\ide{\al}}.
\end{align*}
From this we infer that $(\bv{\theta} - \bv{\phi})\in (\lprp{\!\ide{\al}})^{\perp} = \bide{\al}$ by our last Claim.
But this means that $\| [\theta] \| = \| [\theta_{[\phi]}] \| = \| [\phi] \|$. Therefore, by \eqref{E:isomet} we have
$\| [\phi] \| = \|L_{[\phi]}\|_{{\rm op}}$.
\end{proof}
\medskip

\subsection{A key proposition and Theorem~\ref{T:interp-Char}}\label{SS:interp-Char}

We begin with a proposition that is the key result leading to the proof of Theorem~\ref{T:interp-Char}. It
gives us a way of linking a function $\psi$ belonging to the dual algebra $\dal$, that interpolates the
data $\{(X_j, w_j) : 1\leq j\leq N\}$, to conditions for $\sup_{\D^n}|\psi|$ to be~$\leq 1$. We shall
continue to use the notation introduced in Section~\ref{S:funcThr}, and extend the notation where
needed. For instance
\[
 \IX{\al}^2(g)\,:=\,\text{the closure of $\left.\IX{\al}\right|_{\Tn}$ in  $\Lb^2(\Tn, |g|^2dm)$},
\]
where $g\in\Lb^{\infty}(\Tn)$ and such that $|g|>c_g$ for some constant $c_g > 0$.
\smallskip

We ought to mention that the schema of the proof of the following proposition is that of the proof
of \cite[Theorem~5.13]{mccullough:NPtida96} by McCullough\,---\,with the major difference being the
appearance of $\polcl(R)$. 

\begin{proposition}\label{P:norm_formula}
Let  $X_1,\dots,X_{N}$ be distinct points in $\D^n$. Let $\dal$ be a weak${}^{{\boldsymbol{*}}}$
closed subalgebra of $\hinf$ such that $\dal = \dal(\al)$ for some uniform subalgebra
$\al\subseteq A(\D^n)$ having a tame pre-annihilator. Fix an integer $R\geq 1$, and
let $\polcl(R)\varsubsetneq \C[z_1,\dots, z_n]$ be the class defined in Theorem~\ref{T:interp-Char}.
For any coset $[\phi]\in \dal(\al)/\ide{\al}$
\begin{equation}\label{E:norm_form}
 \| [\phi] \|\,=\,\sup\big\{\left\|\Pi_{\al^2(p)}\circ M^*_{\phi}\circ\spProj{p}\right\|_{{\rm op}} : p\in \polcl(R)\big\},
\end{equation}
where
\begin{align*}
 \Pi_{\al^2(p)}\,&:=\,\text{the orthogonal projection of $\Lb^2(\Tn)$ onto $\al^2(p)$}, \\
 \spProj{p}\,&:=\,\text{the orthogonal projection of $\al^2(p)$ onto $\al^2(p)\ominus \IX{\al}^2(p)$}.
\end{align*}
\end{proposition}
\begin{proof}
Lemma~\ref{L:isom2} suggests that to establish \eqref{E:norm_form} we can work with the
linear functionals $L_{[\phi]}\in (\lprp{\!\ide{\al}}/\lprp{\dal(\al)})^*$. Let us fix a
coset $[f]$. By Proposition~\ref{P:SzegoRep}, we can find
constants $a_1,\dots, a_N\in \C$\,---\,which depend only on the coset $[f]$\,---\,such that
\begin{equation}\label{E:coset_nice}
 [f]\,=\,\bigg[\sum_{1\leq j\leq N} a_jK(X_j, \bcdot)\bigg].
\end{equation}
In what follows (as well as in the next section), we shall use $\|\bcdot\|_1$ to denote the 
$\Lb^1$-norm on $\Lb^1(\Tn)$.
Furthermore, $\qnrm{[f]}$ will denote the quotient norm of $[f]$. Fix an $\eps > 0$. It follows from
\eqref{E:coset_nice} that there exists a function $G_{\eps}\in \lprp{\dal(\al)}$ such that
\[
 \bigg\|\sum_{1\leq j\leq N} a_jK(X_j, \bcdot) + G_{\eps}\bigg\|_1\,<\,\qnrm{[f]} + \eps.
\] 
(It is understood from \eqref{E:coset_nice} that the function
$\sum_{1\leq j\leq N} a_jK(X_j, \bcdot)\in \lprp{\!\ide{\al}}$\,---\,this follows from the reproducing property of the
Szeg{\H{o}} kernel for $H^2(\Tn)\supset \bv{\dal(\al)}$.) By Lemma~\ref{L:lprps} and the
fact that $\al$ has a tame pre-annihilator, we can find a function $H_{\eps}\in 
(\smoo(\Tn; \C)\cap \lprp{\dal(\al)}) $ such that
$\|G_{\eps} - H_{\eps}\|_1 < \eps$. Let us now write:
\[
 F_{\eps}\:=\,\sum_{1\leq j\leq N} a_jK(X_j, \bcdot) + H_{\eps}.
\]
By \eqref{E:coset_nice} and the subsequent discussion, we have:
\begin{itemize}
 \item[$(A)$] $[F_{\eps}] = [f]$;
 \item[$(B)$] $F_{\eps}\in \smoo(\Tn; \C)$;
 \item[$(C)$] $\|F_{\eps}\|_1 < \qnrm{[f]} + 2\eps$.
\end{itemize}
Recall that we have fixed an $R\geq 1$.
Now, $|F_{\eps}| + 3\eps/4$ is a strictly positive continuous function on $\Tn$. Thus, by Proposition~\ref{P:posFn_poly},
there exists a polynomial $\pee\in \polcl(R)$ such that
\begin{equation}\label{E:imp_ineq}
  |F_{\eps}(\zt)|+\eps\,>\,|\pee(\zt)|^2\,>\,|F_{\eps}(\zt)|+\eps/2 \; \; \; \forall \zt\in \Tn.
\end{equation}

In this paragraph, we shall take $g$ to be any function in $A(\D^n)$ such that $\left. g\right|_{\Tn}$ is
non-vanishing. Write
\[
 F_{g}(\zt)\,:=\,\overline{F_{\eps}(\zt)}/|g(\zt)|^2 \; \; \; \forall \zt\in \Tn.
\]
The projection operator $\spProj{g}$ will have a meaning analogous to $\spProj{p}$ defined above.
We note that, owing to the properties of $g$\,---\,and given that by Lemma~\ref{L:clos_new} and
the Claim made in the proof of Lemma~\ref{L:isom2},
$\bv{\dal(\al)}$ and $\bide{\al}\subset (\al^2(1)\cap\Lb^{\infty}(\Tn))$\,---\,we have:
\begin{equation}\label{E:in_hilb}
 \bv{\dal(\al)}\,\subset\,\al^2(g) \quad \text{and} \quad \bide{\al}\subset \IX{\al}^2(g).
\end{equation}
We now compute:
\begin{align*}
 L_{[\phi]}([f])\,&=\,L_{[\phi]}([F_{\eps}]) && (\text{by $(A)$ above}) \\
 &=\,\langle \bv{\phi}\,, F_g\rangle_g \\
 &=\langle \spProj{g}(\bv{\phi}), F_g\rangle_g && (\text{by \eqref{E:in_hilb} and Lemma~\ref{L:isom2}-(1)}) \\
 &=\,\langle \phi,\,\spProj{g}(F_g)\rangle_g \\
 &=\,\langle 1, M^*_{\phi}\circ\spProj{g}(F_g)\rangle_g.
\end{align*}
Hence, we get the useful inequality:
\begin{equation}\label{E:useful}
 |L_{[\phi]}([f])|\,\leq\,\|\Pi_{\al^2(g)}\circ M^*_{\phi}\circ \spProj{g}\|_{{\rm op}}\,\|1\|_g\,\|F_g\|_g,
\end{equation}
which holds true for {\em any} $g$ with the properties stated above. Here $\|\bcdot\|_g$ denotes the
norm on $\al^2(g)$.
\smallskip

At this stage, we shall take $g = \pee$ in \eqref{E:useful}. Since $\pee \in \polcl(R)$, and $R\geq 1$,
$\pee$ has all the properties required of $g$ in the previous paragraph. We ought to state that, after
having chosen $g = \pee$, the rest of the argument for this proof uses the same estimates that
conclude the proof of \cite[Theorem~5.13]{mccullough:NPtida96}. By \eqref{E:imp_ineq}, we have
\[
 |F_{\spee}(\zt)|\,<\,1 \; \; \; \forall \zt\in \Tn.
\]
Therefore, by the last inequality, \eqref{E:imp_ineq} and $(C)$ above, we have:
\begin{align*}
 \|F_{\spee}\|^2_{\spee}&<\,\inttor\!|\pee|^2\,dm\,<\,\qnrm{[f]} + 3\eps, \\
 \|1\|^2_{\spee}&=\,\inttor\!|\pee|^2\,dm\,<\,\qnrm{[f]} + 3\eps.
\end{align*}
Combining the above inequalities with \eqref{E:useful} and letting $\eps\searrow 0$, we get:
\[
 \frac{|L_{[\phi]}([f])|}{\qnrm{[f]}}\,\leq\,\sup\big\{\left\|\Pi_{\al^2(p)}\circ M^*_{\phi}\circ\spProj{p}\right\|_{{\rm op}}
 								: p\in \polcl(R)\big\} \; \; \; \text{if $[f] \neq [0]$}.
\]
Since $[f]$ was chosen arbitrarily, the right-hand side of the above inequality actually dominates
$\|L_{[\phi]}\|_{{\rm op}}$. We now apeal to Lemma~\ref{L:isom2} to get
\[
 \|[\phi]\|\,\leq\,\sup\big\{\left\|\Pi_{\al^2(p)}\circ M^*_{\phi}\circ\spProj{p}\right\|_{{\rm op}} : p\in \polcl(R)\big\}.
\]
The reverse inequality trivially holds true. This establishes \eqref{E:norm_form}. 
\end{proof}
\smallskip

Finally, we present:

\begin{proof}[The proof of Theorem~\ref{T:interp-Char}]
Most of the steps of this proof are similar to those in the proofs of results analogous to Theorem~\ref{T:interp-Char}
in the literature cited in Section~\ref{S:intro}. Hence, we shall be brief. We begin with two very standard facts. For
each $p\in \polcl(R)$.
\begin{itemize}
 \item The set
 $\{\overline{K_{\al,\,p}}(X_1, \bcdot),\dots, \overline{K_{\al,\,p}}(X_N, \bcdot)\}$ spans $\al^2(p)\ominus \IX{\al}^2(p)$.
 \item For any $\phi\in \dal(\al)$, we have
 \[
  M^*_{\phi}(\overline{K_{\al,\,p}}(X_j, \bcdot))\,=\,\overline{\phi(X_j)}\,\overline{K_{\al,\,p}(X_j, \bcdot)}, \; \; \;
  j = 1,\dots, N.
 \]
\end{itemize}
For any $f\in \al^2(p)\ominus \IX{\al}^2(p)$, there exist $c_1,\dots, c_N\in \C$ such that
\[
 f\,=\,\sum_{1\leq j\leq N} c_j\overline{K_{\al,\,p}}(X_j, \bcdot),
\]
whence, we compute:
\[
 \Pi_{\al^2(p)}\circ M^*_{\phi}\circ\spProj{p}(f)\,=\,\sum_{1\leq j\leq N} c_j
 \overline{\phi(X_j)}\,\overline{K_{\al,\,p}(X_j, \bcdot)}.
\]
From this it follows, {\em exactly} (and by an elementary computation) as in several of the works cited in
Section~\ref{S:intro} that:
\begin{multline}\label{E:contractivity}
 \big\|\Pi_{\al^2(p)}\circ M^*_{\phi}\circ\spProj{p}\big\|_{{\rm op}}\,\leq\,1 \\
 \iff\,\left[(1 - \phi(X_j)\overline{\phi(X_k)})\big\langle K_{\al,\,p}(X_j, \bcdot), K_{\al,\,p}(X_k, \bcdot)\big\rangle_{\al^2(p)}
 \right]_{j,\,k = 1}^N\,\geq\,0.
\end{multline}

Now, suppose that there exists a function $F\in \dal(\al)$ such that $F(X_j) = w_j$ for each $j = 1,\dots, N$
and such that $\sup_{\D^n}|F|\leq 1$. This implies that $\|[F]\|\leq 1$. Then,
by Proposition~\ref{P:norm_formula} and \eqref{E:contractivity}, \eqref{E:Pick-sys2} follows.
\smallskip

Conversely, assume \eqref{E:Pick-sys2}. Let $\Phi_1,\dots, \Phi_N\in \al$ be as given by Lemma~\ref{L:0-1lemma}.
Write
\[
 \phi\,:=\,\sum_{1\leq j\leq N}w_j\Phi_j\,\in\,\al.
\]
Observe that $\phi(X_j) = w_j$ for $j = 1,\dots, N$. By \eqref{E:contractivity} and
Proposition~\ref{P:norm_formula}, we get $\|[\phi]\|\leq 1$. From the latter we have, by definition:
\[
 \text{For each $\nu\in \Z_+$, $\exists \psi_\nu\in \ide{\al}$ such that
 $\|\left.\phi\right|_{\Tn} + \bv{\psi}_\nu\|_\infty = \sup\nolimits_{\D^n}|\phi + \psi_\nu| < 1 + \newfrc{1}{\nu}$.}
\]
By Montel's theorem, there exists a sequence $\nu_1 <\nu_2 <\nu_3 <\dots$ and a holomorphic
function $F$ defined on $\D^n$ such that
\[
 \phi + \psi_{\nu_k}\!\lrarw F \; \; \text{uniformly on compact subsets of $\D^n$ as $k\to \infty$.}
\]
By Proposition~\ref{P:closure_subalg}, $F\in \dal(\al)$. Clearly $F(X_j) = w_j$ for $j = 1,\dots, N$, and
$\sup_{\D^n}|F|\leq 1$. 
\end{proof}

\section{The proof of Theorem~\ref{T:non-mod1Set}}\label{S:non-mod1Set_proof}

In this section, it will be assumed throughout that $n\geq 2$.
Before we give a proof of Theorem~\ref{T:non-mod1Set}, let us look at an explicit description of the
space $\lprp{\hinf}$. Write
\[
 \mathbb{Y}^n\,:=\,\Z^n\setminus\nat^n,
\]
where $\nat = \{0, 1, 2,\dots\}$. Then, it is not hard to show that
\begin{equation}\label{E:leftAnnH}
 \lprp{\hinf}\,=\,\text{the closure in $\Lb^1(\Tn)$ of 
 ${\rm span}_{\C}\{\left.\zbar_1^{\alpha_1}\,\zbar_2^{\alpha_2}\dots\zbar_n^{\alpha_n}\right|_{\Tn} : 
 (\alpha_1,\dots, \alpha_n)\in \mathbb{Y}^n\}$}
\end{equation}
(an argument for the above can be found in \cite[Section~5]{mccullough:NPtida96}).
\smallskip

We can now present:
 
\begin{proof}[The proof of Theorem~\ref{T:non-mod1Set}] We shall use
notations analogous to those in Sections~\ref{S:funcThr} and \ref{S:interpCond}. Accordingly, we shall
denote by $\idl$ the following ideal:
\[
 \idl\,:=\,\text{the $\wes$ closure of the set of all $A(\D^n)$-functions that vanish on $\boldsymbol{{\sf X}}$,}
\]
where $\boldsymbol{{\sf X}} = \{X_1,\dots, X_N\}$. We shall, in a very essential way, need to work with
the spaces
\[
 \hinf/\idl \quad \text{and} \quad \lprp{\!\idl}/\lprp{\hinf}.
\]
The notation $\|[\psi]\|$, where $\psi\in \hinf$, will have the same meaning as in Section~\ref{S:interpCond}. Similarly,
$\qnrm{[f]}$ will denote the quotient norm on $\lprp{\!\idl}/\lprp{\hinf}$.
\smallskip

Let $\phi$ be an interpolant in $\hinf$ for the given data. Since, by hypothesis, the data are extremal, we have
\begin{equation}\label{E:norm1}
 \|[\phi]\|\,=\,1.
\end{equation}
We appeal again to Lemma~\ref{L:isom2}. Consider the linear functional
\[
 L_{[\phi]}\,:\,\lprp{\!\idl}/\lprp{\hinf}\ni [f]\,\longmapsto\,\int_{\Tn}\!\!\bv{\phi}f dm.
\]
By \eqref{E:norm1} and Lemma~\ref{L:isom2}, we have
$\|L_{[\phi]}\|_{{\rm op}} = 1$. Furthermore, as $\hinf/\idl$ is finite-dimensional, it follows from 
Lemma~\ref{L:isom2}-$(2)$ that
\begin{equation}\label{E:f_0}
 \exists f_0\in \lprp{\!\idl} \ \text{such that $\qnrm{[f_0]} = 1$ and $\int_{\Tn}\!\!\bv{\phi}f_0 dm = 1$}.
\end{equation}
\pagebreak

\noindent{{\bf Step 1.} {\em Finding ``nice'' coset-representatives for $[f_0]$}}
\vspace{0.5mm}

\noindent{By Proposition~\ref{P:SzegoRep}\,---\,taking $\al = A(\D^n)$, whence $\dal(\al) = \hinf$\,---\,there
exist constants $a_1,\dots, a_N\in \C$\,---\,not all of which
are $0$\,---\,such that
\[
 [f_0]\,=\,\bigg[\sum_{1\leq j\leq N} a_jK(X_j, \bcdot)\bigg]
\]
(Recall that, by the reproducing property, $K(X_j, \bcdot)\in \lprp{\!\idl}$ for each $j = 1,\dots, N$.)
By definition
\[
 \qnrm{[f_0]}\,:=\,\inf\{\|f_0 + g\|_{1} : g\in \lprp{\hinf}\}.
\]
So, if we fix $\eps > 0$, there exists a function $g_{\eps}\in \lprp{\hinf}$ such that
\begin{equation}\label{E:qNorm-L1}
 \bigg[\sum_{1\leq j\leq N} a_jK(X_j, \bcdot) + g_{\eps}\bigg] = [f_0] \quad \text{and}
 \quad 1\leq \left\|\sum_{j=1}^N a_jK(X_j, \bcdot) + g_{\eps}\right\|_1 < 1 + \eps/2.
\end{equation}
From the brief discussion prior to this proof, 
\eqref{E:leftAnnH} in particular, it follows that
there exists a polynomial $P_{\eps}$, in $z$ and $\zbar$, of the form
\[
 P_{\eps}(z)\,=\,\sum_{\alpha\in {\sf F}(\eps)}C_{\alpha}\zbar_1^{\alpha_1}\,\zbar_2^{\alpha_2}\dots\zbar_n^{\alpha_n},
\]
where ${\sf F}(\eps)$ is a finite subset of $\mathbb{Y}^n$, such that
\begin{equation}\label{E:approx2-P}
 \|\left.P_{\eps}\right|_{\Tn} - g_{\eps}\|_1\,<\,\eps/2.
\end{equation}
By the form of the polynomial $P_{\eps}$, we see that $P_{\eps}\in \lprp{\hinf}$. Thus, by \eqref{E:qNorm-L1}
and \eqref{E:approx2-P}, we have
\begin{equation}\label{E:reg}
 [f_0]\,=\,\bigg[\sum_{1\leq j\leq N} a_jK(X_j, \bcdot) + \left.P_{\eps}\right|_{\Tn}\bigg] \quad \text{and}
 \quad 1\leq \left\|\sum_{j=1}^N a_jK(X_j, \bcdot) + \left.P_{\eps}\right|_{\Tn}\right\|_1 < 1 + \eps.
\end{equation}}

Let us write
\[
 G_{\eps}\,:=\,\sum_{1\leq j\leq N} a_jK(X_j, \bcdot) + \left.P_{\eps}\right|_{\Tn}.
\]
Let us emphasise how regular $G_{\eps}$ is. Note, firstly, that for each $X_j\in \boldsymbol{{\sf X}}$,
$\overline{K(X_j, \bcdot)}$ is holomorphic (in its second variable) in some neighbourhood\,---\,which depends
on $X_j$\,---\,of $\overline{\D^n}$. Now define the function $\gamma_{\eps}$ which is holomorphic on
${\sf Ann}(0; 1\pm\delta)^n$\,---\,where $\delta>0$ is determined by $X_1,\dots, X_N$\,---\,as follows:
\[
 \gamma_{\eps}(z)\,:=\,\sum_{1\leq j\leq N} \overline{a}_j\overline{K(X_j, z)} + 
 	\sum_{\alpha\in {\sf F}(\eps)}\overline{C}_{\alpha}\prod_{j=1}^nz^{\alpha_1} \; \; \;
 	\forall (z_1,\dots, z_n)\in {\sf Ann}(0; 1\pm\delta)^n.
\]
Observe that
\begin{equation}\label{E:why-reg}
 \left.\overline{\gamma}_{\eps}\right|_{\Tn}\,=\,G_{\eps}.
\end{equation}
In short, associated to $[f_0]$ is a family of coset-representatives $G_\eps$ that are restrictions
to $\Tn$ of antiholomorphic functions and whose $\Lb^1$-norms decrease to $1$.
\smallskip

\noindent{{\bf Step 2.} {\em Finding a sequence of measures with useful properties}}
\vspace{0.5mm}

\noindent{Since $\gamma_{\eps}\in \hol({\sf Ann}(0; 1\pm\delta)^n)$, it follows from \eqref{E:why-reg} that
$G_{\eps}^{-1}\{0\}$ is a real-analytic subset of $\Tn$. As $G_{\eps}\not\equiv 0$, it follows from the basic
theory of real-analytic sets that
\begin{equation}\label{E:zero-meas}
 m(G_{\eps}^{-1}\{0\})\,=\,0 \; \; \; \text{for each $\eps>0$}.
\end{equation}}
\vspace{-2mm}

Let us now define the positive measures $\mu_{\eps}$ on $\Tn$ such that $d\mu_{\eps} = |G_{\eps}|dm$.
These measures have the following useful property:
\begin{align}
 \mu_{\eps}\!\left[\{\zt\in \Tn : 1 - |\bv{\phi}(\zt)| \geq \sqrt{\eps}\}\right]\,&\leq\,\frac{1}{\sqrt{\eps}}
 \int_{\Tn}(1-|\bv{\phi}|)|G_{\eps}|dm \notag \\
 &<\,\frac{1}{\sqrt{\eps}}\bigg((1+\eps) - \left|\int_{\Tn}\bv{\phi}G_{\eps}dm\right|\bigg)\,=\,\sqrt{\eps},
 \label{E:small-meas}
\end{align}
which follows from Chebyshev's inequality, \eqref{E:f_0} and \eqref{E:reg}.
\smallskip

We would ultimately like to estimate the {\em Lebesgue} measures of the above sets. To that end, we have the
following observation. Write
\[
 \Gamma_{\eps}(\zt)\,:=\,\begin{cases}
 						1/|G_{\eps}(\zt)|, &\text{if $\zt\notin G_{\eps}^{-1}\{0\}$}, \\
 						0, &\text{if $\zt\in G_{\eps}^{-1}\{0\}$}.
 						\end{cases}
\]
Clearly, $\Gamma_{\eps}\in \Lb^{1}(\Tn, d\mu_{\eps})$ for each $\eps>0$. It follows from \eqref{E:zero-meas}
that
\begin{equation}\label{E:relate}
 m(E)\,=\,\int_E\Gamma_{\eps}d\mu_{\eps} \; \; \; \text{for every Lebesgue measurable set $E\subseteq \Tn$}
\end{equation}
for each $\eps>0$.
\smallskip

\noindent{{\bf Step 3.} {\em Completing the proof}}
\vspace{0.5mm}

\noindent{Recall that $\bv{\phi}$ is undefined on a set of Lebesgue measure zero. It will not affect the conclusions
of the argument below if we fix $\bv{\phi}(\zt) = 0$ on the latter set. Note that
\begin{equation}\label{E:dec-int}
 \{\zt\in \Tn : 1 - |\bv{\phi}(\zt)| > 0\}\,=\,\liminf_{k\to\infty}\{\zt\in \Tn : 1 - |\bv{\phi}(\zt)| > 1/k^3\}.
\end{equation}
Denote the set on the left-hand side of the above equation as $\mathcal{S}$ and write
$E_k := \{\zt\in \Tn : 1 - |\bv{\phi}(\zt)| > 1/k^3\}$. Let us define
\begin{align*}
 A_k\,&:=\,\{\zt\in E_{k} : |G_{1/k^6}(\zt)| \geq 1/k\}, \\
 B_k\,&:=\,\{\zt\in E_{k} : |G_{1/k^6}(\zt)| < 1/k\}, \; \; \; k = 1, 2, 3,\dots
\end{align*}
From \eqref{E:small-meas}, we have $\mu_{1/k^6}(A_k) < 1/k^3$. Thus, from \eqref{E:relate}, we get
\begin{equation}\label{E:A_k-meas}
 m(A_k)\,=\,\int_{A_k}\frac{1}{|G_{1/k^6}|}d\mu_{1/k^6}\,\leq\,\frac{1}{k^2} \; \; \; \forall k\in \Z_+.
\end{equation}
Let us define $S_k := \{\zt\in \Tn : |G_{1/k^6}(\zt)| < 1/k\}$. Then
\[
 S_k\,=\,\{\zt\in \Tn : -\overline{G_{1/k^6}(\zt)}G_{1/k^6}(\zt) > -1/k^2\}, \; \; \; k = 1, 2, 3,\dots,
\]
whence, by \eqref{E:why-reg}, each $S_k$ is semi-analytic. And clearly, as $\|G_{1/k^6}\|_{1}\geq 1$,
each $S_k$ is a proper subset of $\Tn$.}
\smallskip

\noindent{{\bf Claim.} $\mathcal{S}\subseteq \limsup_{k\to \infty}A_k \cup \liminf_{k\to \infty}S_k$.}\\
\vspace{.5mm}

\noindent{Pick a $\zt\in \mathcal{S}$. Then, $\exists k_1(\zt)\in \nat$ such that $\zt\in E_k \ \forall k\geq k_1(\zt)$.
Suppose $\zt\notin \limsup_{k\to \infty}A_k$. By definition, $\exists k_2(\zt)\in \nat$ such that $\zt\notin A_k \ \forall
k \geq k_{2}(\zt)$. As $A_k$ and $B_k$ partition $E_k$, it follows that
\[
 \zt\in B_k \subseteq S_k \; \; \; \forall k\geq \max(k_1(\zt), k_2(\zt)).
\]
The claim follows.}
\smallskip

Recall that $m$ is normalized to be a probability measure. Thus, by \eqref{E:A_k-meas} and the Borel--Cantelli
lemma, we have $m(\limsup_{k\to \infty}A_k) = 0$. Finally, let us write:
\[
 N\,:=\,\limsup_{k\to \infty}A_k \quad\text{and} \quad
 S\,:=\,\liminf_{k\to \infty}S_k.
\]
Since $A_k\cap S_k = \varnothing$ $\forall k\in \Z_+$, it is very easy to see that $S\cap N = \varnothing$. Thus
$\mathcal{S}\subset N\sqcup S$.
\end{proof}

\end{document}